\documentclass[10pt,oneside,english]{article}
\usepackage[english]{babel}
\usepackage[utf8]{inputenc}
\usepackage{graphicx}
\usepackage{graphics}
\usepackage{geometry}
\usepackage{amssymb}
\usepackage{amsmath, amsthm}
\usepackage{amsmath,amscd}
\usepackage[all,cmtip]{xy}
\usepackage{hyperref}
\usepackage{color}
\usepackage{float}
\usepackage{listings}
\usepackage{subfig}
\usepackage{bm}
\usepackage[final]{pdfpages}
\usepackage{fancyhdr}
\usepackage{epstopdf}
\usepackage{url}
\usepackage{verbatim}

\usepackage{graphics}
\usepackage{enumerate}
	\usepackage{fancyhdr}
	\usepackage{time}

	\newcommand\myTime{\now}
\makeatletter
 \theoremstyle{plain}
 \newtheorem{thm}{Theorem}[section]
 \newtheorem{prop}[thm]{Proposition} 
 \newtheorem{lem}[thm]{Lemma} 

 \theoremstyle{definition}
 \newtheorem{defn}[thm]{Definition}

 \theoremstyle{remark}
 \newtheorem{rem}[thm]{Remark}

\newcommand {\lpair} {\langle}
\newcommand {\rpair} {\rangle}

\newcommand{\eps}{\varepsilon} 
\newcommand{\R}{\mathbb{R}}    
\newcommand{\N}{\mathbb{N}}    

\newcommand{\J}{\mathbf{J}}    
\newcommand{\tr}{\operatorname{tr}\,}

\newcommand{\Id}{\text{Id}}    
\DeclareMathOperator{\Ker}{Ker}  

\newcommand{\restr}[1]{\,\vrule height3ex width.4pt depth1.4ex\lower1.4ex\hbox{\scriptsize $\,#1\,$}}
\newcommand{\rrestr}[1]{\,\vrule height2ex width.4pt depth0.9ex\lower0.9ex\hbox{\scriptsize $\,#1\,$}}
\newcommand{\ed}{\mathbf{d}}   

\newcommand{\Ad}{\mathrm{Ad}}  
\newcommand{\ad}{\mathrm{ad}}  
\newcommand{\D}{\mathbf{D}}    

\newcommand {\g}{\mathfrak{g}}
\newcommand {\h}{\mathfrak{h}}
\newcommand {\q}{\mathfrak{q}}
\newcommand {\p}{\mathfrak{p}}
\renewcommand {\q}{\mathfrak{q}}
\renewcommand {\k}{\mathfrak{k}}
\renewcommand {\l}{\mathfrak{l}}
\newcommand{\s}{\mathfrak{s}}
\newcommand{\m}{\mathfrak{m}}
\newcommand{\n}{\mathfrak{n}}
\renewcommand{\o}{\mathfrak{o}}

\newcommand{\Orb}{\mathcal{O}}  

\title{The Hamiltonian Tube Of A\\ Cotangent-Lifted Action}
\author{Miguel Rodríguez-Olmos\thanks{miguel.rodriguez.olmos@ma4.upc.edu}, Miguel Teixidó-Román\thanks{miguel.teixido@ma4.upc.edu}}

\date{Department of Applied Mathematics IV\\ Technical University of Catalonia. Barcelona, Spain.}

\begin{document}
\maketitle

\begin{abstract}

The Marle-Guillemin-Sternberg (MGS) form is local model for a neighborhood of an orbit of a Hamiltonian Lie group action on a symplectic manifold. One of the main features of the MGS form is that it puts simultaneously in normal form the existing symplectic structure and  momentum map. The main drawback of the MGS form is that it does not have an explicit expression. We will obtain a MGS form for cotangent- lifted actions on cotangent bundles that, in addition to its defining features, respects the additional fibered structure present. This model generalizes previous results obtained by T. Schmah for orbits with fully-isotropic momentum. In addition, our construction is explicit up to the integration of a differential equation on $G$. This equation can be easily solved for the groups $SO(3)$ or $SL(2)$, thus giving explicit symplectic coordinates for arbitrary canonical actions of these groups on any cotangent bundle.

\vspace{4mm}
\noindent MSC 2010: 53D20; 70H33; 37J15

\end{abstract}

\tableofcontents

\section{Introduction}
The study of the local geometry of symplectic manifolds equipped with Hamiltonian group actions constitutes a field originated with the classical papers \cite{MR859857,guillemin1984normal}.
In these references the authors obtain a universal model for a tubular neighborhood of the orbit of a point under a Hamiltonian action which puts in normal form both the symplectic structure and the momentum map (Theorem \ref{thm:classicalMGS}). This model is known as the \emph{Hamiltonian tube} or \emph{Marle-Guillemin-Sternberg form} and it is the base of almost all the local studies concerning Hamiltonian actions of Lie groups on symplectic manifolds. In fact, since the decade of the 80's almost all the relevant results about the qualitative local dynamics of equivariant Hamiltonian flows have been obtained using techniques based on the Hamiltonian tube, for example \cite{MR1660367,MR1150395,montaldi1997persistence,MR1656378,MR946385,MR1958531,MR1614032,MR1690200,patrick1995relative,MR1483562,MR1885680,simo1991stability}. Moreover, the Hamiltonian tube was a major ingredient for many of the generalizations of the Marsden-Weinsten reduction scheme to singular actions like \cite{MR1486529,ortegaratiu,MR1127479,
MR1809496}.

In this paper we will focus on the cotangent bundle case. Let $T^* Q$ be a cotangent bundle equipped with its canonical symplectic structure $\omega_Q$ and let $G$ be a Lie group that acts smoothly on $Q$. The canonical lift of this action to $T^*Q$ is automatically a Hamiltonian action.
The Marle-Guillemin-Sternberg construction (Theorem \ref{thm:classicalMGS}) applied to a canonical action over a cotangent bundle gives, as for every Hamiltonian action, an equivariant local model of  $(T^*Q,\omega_Q)$
that puts in normal form both the symplectic structure and the momentum map. However, in general this model does not respect the fibration $T^*Q\to Q$. Moreover, the map given by Theorem \ref{thm:classicalMGS} is not constructive and only some of its properties are known.  In the concrete case of cotangent bundles there is a strong motivation coming from  geometric mechanics and geometric quantization that makes desirable to obtain explicit or fibrated local models.
In this paper we obtain a construction of the Hamiltonian tube for a canonical cotangent-lifted action in a cotangent bundle specially adapted to this kind of manifolds and that puts the fibration in a normal form (Theorem \ref{thm:tubegeneral}).
In other words, this assumes that the space that models locally the neighborhood of an orbit of the group in $T^*Q$ has a fibered structure $\tau:Y\rightarrow Q$
where  $U$ is a local model of the base $Q$ such that $$\mathcal{T}: Y\rightarrow T^*Q$$ is a fibered map.
Additionally, the construction of $\mathcal{T}$ will be explicit up to the integration of a differential equation on $G$. The restricted $G$-tubes (Definition \ref{def:restrictedtube}) will be the basic building blocks and are the only non-explicit part of the model. Given a Hamiltonian action the restricted $G$-tube depends only on the group $G$ and its algebraic structure. For example, for $SO(3)$ and $SL(2)$ the expressions of their restricted $G$-tubes can be obtained explicitly, see Section \ref{sec:explicitexamples}.
 For larger groups the computation will be more cumbersome but could be done with a computer algebra system. As an additional result of our construction  we obtain a fibered analogue of the  Lerman-Bates lemma \cite{MR1486529} (Proposition \ref{prop:Bates-Lerman}) that characterizes the set of points with prescribed momentum  in a neighborhood of the form $\tau^{-1}(U)$. We believe that this result can be used to study in detail the structure of singular reduction for cotangent bundles, generalizing the results of \cite{MR2271206} to non-zero momentum. This will be addressed elsewhere.

The first works studying symplectic normal forms in the specific case of cotangent bundles seem to have been \cite{phdschmah, MR2293645}. In these references T. Schmah found a Hamiltonian tube around those points $z\in T^*Q$ such that its momentum $\mu=\J(z)$ is fully-isotropic (that is, $G_\mu=G$ with respect to the coadjoint representation). One of the main differences between her construction and the classical MGS model for symplectic actions is that the one for cotangent bundles was constructive, unlike the general MGS model. The next step came with \cite{MR2421706} where it is provided a general descripion of the symplectic slice of a cotangent bundle, without the assumption $G=G_\mu$. Recently \cite{1311.7447} constructed Hamiltonian tubes for free actions of a Lie group $G$ and showed that this construction can be made explicit for $G=SO(3)$.

This paper is organized as follows: in Section \ref{sec:preliminaries} we review some background material regarding proper actions and the classical MGS model. In Section \ref{sec:symplslice} we introduce in Proposition \ref{prop:liealgsplitting} a splitting of the Lie algebra that will be needed in all the subsequent development. This splitting already appeared partially in \cite{MR2421706}  (Theorem \ref{thm:symplnorm-cotang}). In Section \ref{sec:gtubes} we introduce simple and restricted $G$-tubes (Definitions \ref{def:simpletube} and \ref{def:restrictedtube}). Simple $G$-tubes  are, up to technical details, MGS models for the lift of the left action of $G$ on itself to $T^*G$. Their existence is proved in Proposition \ref{prop:simpleGtube}.
 Restricted $G$-tubes are defined implicitly in terms of a simple $G$-tube (Proposition \ref{prop:restrtubes}) and are the technical tool that we will need later to construct the general Hamiltonian tube.

In Section \ref{sec:cotangenttubes} we construct the general Hamiltonian tube for a cotangent-lifted action in such a way that it is explicit up to a restricted $G$-tube. This general tube will be the composition of two maps. We will first construct a Hamiltonian tube around points in $T^*Q$ with certain maximal isotropy properties (Theorem \ref{thm:tube0}) and then an adaptation of the ideas of \cite{MR2293645} will be used to construct a $\Gamma$ map (Proposition \ref{prop:Gamma}). Together these two maps will give the general Hamiltonian tube in Theorem \ref{thm:tubegeneral}.
In Section \ref{sec:Bates-Lerman} we use a zero section-centered tube to construct a cotangent-bundle version of a result due to Bates and Lerman in Proposition \ref{prop:fiberedBL}. One important novelty of this lemma is that our version is global in the vertical direction in the sense that given $z\in \J^{-1}(\mu)$ we can describe, via a Hamiltonian tube, not only a set of the form $\J^{-1}(\mu)\cap U$ where $U$ is a neighborhood of $z$, but a set $\tau^{-1}(\tau(U))\cap \J^{-1}(\mu)$ with $U$ is a neighborhood of $z$. That is, with one Hamiltonian tube we can describe all the points in $\J^{-1}(\mu)$ whose projection is close enough to the projection of the center point.
Finally, in Section \ref{sec:explicitexamples} we present explicit examples of $G$-tubes for both the groups $SO(3)$ (where we recover the results of \cite{1311.7447}) and $SL(2,\R)$. In Subsection \ref{sec:so3onr3} we present an explicit Hamiltonian tube for the natural action of $SO(3)$ on $T^*\R^3$ which generalizes the final example of \cite{MR2293645} to the case $\mu\neq 0$.

\paragraph{Acknowledgements.}
The authors would like to acknowledge the financial support of the Ministerio de Ciencia e Innovaci\'on (Spain), project
MTM2011-22585 and AGAUR, project 2009 SGR:1338. M. Teixid\'o-Roma\'an also thanks the support of a FI-Agaur PhD Fellowship. M. Rodr\'{\i}guez-Olmos thanks the support of the EU-ERG grant ``SILGA".

\section{Preliminaries}
\label{sec:preliminaries}
This section collects  background material from the theory of Hamiltonian actions and MGS normal forms that will be used through this paper. Most material is standard and can be found in greater detail in several references, for instance \cite{ortegaratiu}.
\subsection{Proper actions and slices}

Let $G$ be a Lie group with Lie algebra $\g$. We will always denote by $e$ the identity element of the group and by $L_g,R_g;G\to G$ the left and right multiplications by $g$ respectively. If $G$ acts on $M$ we say that $M$ is a $G$-space. For $p\in M$ the \emph{isotropy subgroup of $p$} is $$G_p=\{g\in G\  |\  g\cdot p = p\}.$$ A map $f:M_1\to M_2$ between two manifolds endowed with $G$-actions is called \emph{$G$-equivariant} if $f(g\cdot p)=g\cdot f(p)$ for all $p\in M_1$ and $g\in G$.

An action is \emph{proper} if the map $G\times M\to M\times M$ defined by $(g,z)\mapsto (z,g\cdot z)$ is a proper map. For a proper $G$-action all the isotropy subgroups $G_z$ are compact subgroups of $G$. For any Lie group $G$ the left and right actions on itself are proper.

\label{sec:twisted}

If a compact subgroup $H\subset G$ acts on a manifold $A$ then on $G\times A$ we can consider two actions:
\begin{itemize}
 \item \emph{twisting action} of $H$: $h\cdot^T (g,a)=(gh^{-1},h\cdot a),\quad h\in H$
 \item \emph{left action} of $G$: $h\cdot^L (g,a)=(hg,a),\quad h\in G$.
\end{itemize}
If necessary we will use as above the superindexes $T$ or $L$ to indicate the $H^T$-action (twisting) or the $G^L$-action (left) on the product $G\times A$. As both actions commute  $G\times A$ supports an action of the direct product group $G^L\times H^T$.

The twisting action is free and proper and, therefore, the quotient space $(G\times A)/H^T$ is a manifold and it will be called the \emph{twisted product}. We will denote it as $G\times_H A$ and its elements will be denoted as $[g,a]_H\quad g\in G,\ a\in A$. The twisted product $G\times_H A$ admits a proper $G$-action given by $g\cdot[g',a]_H=[gg',a]_H$. In fact the twisted product is exactly the associated bundle for the principal $H$-bundle $G\to G/H$ and the $H$-manifold $A$.

The Tube Theorem \cite{MR0126506} shows that in fact every proper $G$-space is locally a twisted product. That is, if $G$ acts properly on $M$, $z\in M$ and $A$ is an $G_z$-invariant complement
 of $g\cdot z$ in $T_zM$ then there is a $G$-equivariant diffeomorphism
\begin{equation}
\mathbf{s}:G\times_{G_z} A \longrightarrow U\subset  M \label{eq:Palais}
\end{equation}
defined on a neighborhood of the zero section of the twisted product satisfying $\mathbf{s}([e,0]_{G_z})=z$ where $U$ is a $G$-invariant neighborhood of $z$.

A $G_z$-invariant complement of $\g\cdot z$ in $T_zM$ will be called a \emph{linear slice} for the $G$ action at $z\in M$.

\subsection{Hamiltonian actions and Hamiltonian tubes}
Assume  now that $G$ acts symplectically on a symplectic manifold $(M,\omega)$. A momentum map is a function $\J:M\to \g^*$ such that $$i_{\xi_M}\omega=\ed\langle\mathbf J(\cdot),\xi\rangle \quad \forall \xi \in \g$$ where $\xi_M\in \mathfrak{X}(M)$ is the fundamental vector field associated with $\xi\in \g$. If $\J$ is equivariant with respect to the coadjoint action on $\g^*$ then we will say that the action is \emph{Hamiltonian}.
If $G$ acts Hamiltonially on a symplectic manifold $(M,\omega)$ there is a symplectic version of the Tube Theorem for proper $G$-spaces, and this is precisely the content of the \emph{Marle-Guillemin-Sternberg normal form} proven by Marle, Guillemin and Sternberg in \cite{MR859857,guillemin1984normal} for compact groups and extended to proper actions of arbitrary groups in \cite{MR1486529}.

\begin{thm}[Hamiltonian Tube Theorem]
Let $(M,\omega)$ be a symplectic manifold endowed with a proper Hamiltonian action of a  Lie group $G$ with momentum map $\J:M\rightarrow \g^*$. Let $z\in M$,
 $\mu=\J(z)$, and choose a $G_z$-invariant splitting $\g_\mu=\g_z\oplus \m$. Let $N$ be a $G_z$-invariant  complement of $\g_\mu\cdot z$ in $\operatorname{Ker} T_z\J$ and $\J_N:N\rightarrow \g_z^*$ defined by $\langle\J_N(v),\xi\rangle:=\frac{1}{2}\omega(\xi\cdot v,v)$. Consider the set $Y:=G\times_{G_z}(\m^*\times N)$ equipped with the two-form
\begin{equation}
 \begin{split}
  \Omega_Y(T_{(g,\nu,v)}\pi_{{G_z}}(u_1),T_{(g,\nu,v)}\pi_{{G_z}}(u_2))=
  \langle \dot \nu_2+T_{v}\J_N(\dot v_2),\xi_1\rangle
  -\langle \dot \nu_1+T_{v}\J_N(\dot v_1),\xi_2\rangle+\\
  +\langle \nu+\J_N(v)+\mu,[\xi_1,\xi_2]\rangle+\omega(\dot v_1,\dot v_2)
 \end{split}
 \label{eq:MGSsymp}
\end{equation}

where $u_i=(T_eL_g\xi_i;\dot \nu_i,\dot v_i) \in T_{(g,\nu,v)}(G\times \m^*\times N)$ and $\pi_{G_z}:G\times(\m^*\times N)\to G\times_{G_z}(\m^*\times N)$. There is a neighborhood $Y_r$ of the zero section of $Y$ such that the restriction $(Y_r,\Omega_{Y})$ is a symplectic manifold equipped with a Hamiltonian action of $G$ (seen as a twisted product) for which the momentum map is
\begin{equation}\label{eq:MGS-momentum}
\J_{Y}[g,\nu,v]_{G_z} = \Ad_{g^{-1}}^*(\mu+\nu+\J_N(v)).
\end{equation}

Additionally, there is a map
\[\mathcal{T}: Y_r \longrightarrow M\]
such that:
\begin{itemize}
  \item $\mathcal{T}: Y_r\longrightarrow \mathcal{T}(Y_r)\subset M$ is a $G$-equivariant diffeomorphism with $\mathcal{T}([e,0,0]_{G_z})=z$.
  \item $\mathcal{T}^* \omega=\Omega_{Y}$.
\end{itemize}

The pair $(Y_r,\Omega_Y)$ is called the \emph{MGS model at $z\in M$}, the $G$-equivariant symplectomorphism $\mathcal{T}$ is called a \emph{Hamiltonian tube around $z$} and the space $N$ a \emph{symplectic slice} at $z$.
\label{thm:classicalMGS}
\end{thm}

Note that whereas in the Palais model \eqref{eq:Palais} the twisted product depends only on a $G_z$-invariant complement to $\g\cdot z$ in $T_z M$ in the Hamiltonian tube the normal form depends on a $G_z$-invariant complement $N$  of $\g_\mu\cdot z$ in $\operatorname{Ker} T_z\J$ and on the complement $\m$ of $\g_z$ in $\g_\mu$.
Note also that $N$, equipped with the restriction of $\omega(z)$ is a symplectic linear space supporting a linear Hamiltonian representation of $G_z$ which admits $\J_N$ as equivariant momentum map.
In addition, since $\mathcal{T}$ is a symplectomorphism equivariant with respect to the Hamiltonian actions of $G$ on $Y_r$ and $M$, by general geometric arguments we have
$$\J_Y=\J\circ\mathcal{T}$$
when the above expression is well defined.

\subsection{The MGS model as a reduced space}
\label{sec:reduced-MGS}

We can interpret the symplectic form $\Omega_Y$  in \eqref{eq:MGSsymp} as the reduced symplectic form for a more basic structure. In this section we are going to recall some well-known facts about this interpretation of the MGS model that will be used throughout the paper.
Let $G$ be a Lie group, $\mu\in \g^*$ and $K\subset G_\mu$ a compact subgroup. Since $K$ is compact we can choose a $K$-invariant complement of $\g_\mu$ in $\g$ and this choice induces a $K$-equivariant linear inclusion $\iota:\g_\mu^*\to \g^*$.
 Consider the product $T_\mu:=G\times \g_\mu^*$ and the map $T_\mu\rightarrow T^*G$ given by  $(g,\nu)\mapsto T_eL_{g^{-1}}^*(\mu+\iota(\nu))\in T^*G$.
 With this map we can  pull-back the canonical symplectic form of $T^*G$
 obtaining the  two-form $\omega_{T_\mu}$ given by
 \begin{equation*}\omega_{T_\mu}(g,\nu)(v_1,v_2)
 =\lpair\dot \nu_2,\xi_1\rpair-\lpair\dot \nu_1,\xi_2\rpair
  +\lpair \mu+\iota(\nu),[\xi_1,\xi_2]\rpair, 
  \end{equation*}
  where $v_i=(T_eL_g\xi_i,\dot \nu_i)\in T_{(g,\nu)}G\times \g_\mu^*$. This form satisfies $\omega_{T_\mu}=-\ed \theta_{T_\mu}$ where
 \begin{equation}
\theta_{T_\mu}(g,\nu)(v_1) = \lpair \mu+\iota(\nu),\xi_1\rpair  . \label{eq:Tmu-potential}  \end{equation}

It can be checked that, for any $g\in G$ the two-form  $\omega_{T_\mu}(g,0)$ is non-degenerate and, therefore there is  an open $K$-invariant neighborhood $(\g_\mu^*)_r$ of $0\in \g_\mu^*$ such that $(G\times (\g_\mu^*)_r,\omega_{T_\mu})$ is a symplectic space (see Proposition 7.2.2 of \cite{ortegaratiu}).
 Let $(N,\omega_N)$ be a  symplectic linear space with a $K$-Hamiltonian linear action with momentum map \begin{equation}\label{linearmomentum}\langle\J_N(v),\xi\rangle:=\frac{1}{2}\omega_N(\xi\cdot v,v).\end{equation} The product $Z:=G\times((\g_\mu^*)_r\times N)$ equipped with the two-form $\omega_{T_\mu}+\omega_N$ is a symplectic space and the natural $G^L$ and $K^T$-actions are free and Hamiltonian with momentum maps
 \begin{equation}\mathbf{K}_{K^T}(g,\nu,v)=-\nu\rrestr{\k}+\J_N(v),\quad \text{and}\quad\mathbf{K}_{G^L}(g,\nu,v)=\Ad_{g^{-1}}^*\nu. \label{eq:Zmomentummaps}\end{equation}
 By the Marsden-Weinstein reduction procedure \cite{MR0402819} the quotient $\mathbf{K}_{K^T}^{-1}(0)/K^T$ is a symplectic manifold.  Since the $G^L$ and $K^T$ actions commute then the induced $G$-action on this quotient is also Hamiltonian.

Let now $\m$ be a $K$-invariant complement of $\k=\text{Lie}(K)$ in $\g_\mu$. There are small enough open neighborhoods $\m^*_r$ and $N_r$ of the origin in $\m^*$ and $N$ such that $\nu+\J_N(v)\in (\g_\mu^*)_r$ for every  $\nu\in \m^*_r$ and  $v\in N_r$. In this setting
\begin{align}
L:G\times_K((\m^*)_r\times N_r) &\longrightarrow \mathbf{K}_{K^T}^{-1}(0)/K^T  \label{eq:Lmap}\\
[g,\nu,v]_K &\longmapsto [g,\nu+\J_N(v),v]_K \nonumber
\end{align}
is a well-defined $G$-equivariant symplectomorphism between the MGS model $Y_r=G\times_K(\m_r^*\times N_r)$ equipped with the symplectic form \eqref{eq:MGSsymp} and the reduced space $\mathbf{K}_{K^T}^{-1}(0)/K^T$ equipped with its canonical reduced symplectic form.

\section{The symplectic slice for cotangent-lifted actions}
\label{sec:symplslice}
If $G$ acts on a manifold $Q$ then the natural lift of the action to $T^*Q$ will be called the \emph{cotangent-lifted action}. The cotangent-lifted action of a proper action is again proper and it is Hamiltonian with respect to the canonical symplectic form with momentum map
\begin{equation}\langle\mathbf J_{T^*Q}(q,p),\xi\rangle =\langle p,\xi_Q(q)\rangle \label{eq:TQmomentum}\end{equation}

Fix $z=(q,p)\in T^*Q$ with $\mu:=\J_{T^*Q}(z)$. By the equivariance of the cotangent bundle projection $\tau:T^*Q\to Q$ and of $\J_{T^*Q}$ we have   $G_z\subset G_q$ and $G_z\subset G_\mu$. Moreover, from \eqref{eq:TQmomentum} we see that $\mu$ is annihilated by any $\xi\in \g_q$. Graphically, the several groups involved can be put into lattice form as
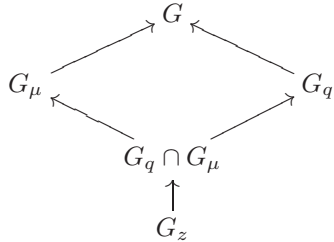
\[\xymatrixrowsep{1pc}\xymatrix{
      & G &  \\
G_\mu \ar[ur]  &   & G_q \ar[ul] \\
      & G_q\cap G_\mu \ar[ul] \ar[ur] & \\
      & G_z\ar[u] & \\}\]
where each arrow represents an inclusion.
As we shall see, one of the main problems of building a Hamiltonian tube for a cotangent-lifted action is that the relationship between the subgroups $G_q$, $G_z$ and $G_\mu$ can be complicated in general. In this section we will first introduce a splitting of $\g$ that has good properties with respect to $\mu$ and $G_q$ and then will use this splitting to restate the characterization of the symplectic slice for cotangent-lifted actions obtained in \cite{MR2421706}.
\subsection{A Lie algebra decomposition}
We start by giving a useful invariant splitting of $\g$ that will be the starting point of many of the constructions we will make in the next sections.

\begin{prop}

 Let $G$ be a Lie group, $H$ a compact subgroup and $\mu\in \g^*$ with $[\h,\h]\in \langle\mu\rangle^\circ$. Let $\Omega^\mu$ be the bilinear form on $\g$ given by $\Omega^\mu(\xi_1,\xi_2)=-\langle \mu,[\xi_1,\xi_2]\rangle$. Then there is a $H_\mu$-invariant splitting \begin{equation}\g=\g_\mu\oplus \o \oplus \l\oplus \n \label{eq:g-splitting}\end{equation} satisfying
 \begin{enumerate}
  \item $\h=\h_\mu\oplus \l$.
  \item $\Omega^\mu\restr {\o}$ is non-degenerate.
  \item $\l$ and $\n$ are $\Omega^\mu$-isotropic subspaces and $\Omega^\mu\restr {\l\oplus \n}$ is non-degenerate.
      \item $\Omega^\mu(\xi_1,\xi_2)=0\quad\text{if}\quad \xi_1\in\o,\,\xi_2\in\l\oplus\n$.
 \end{enumerate}
\label{prop:liealgsplitting}
\end{prop}

\begin{proof}
As $H$ is  compact  we can endow $\g$ with a $\Ad_H$-invariant metric.  The two-form $\Omega^\mu$ restricted to $\g_\mu^\perp$ is non-degenerate because if
$\xi\in \Ker \Omega^\mu\restr{\g_\mu^\perp}$ then
$\lpair \mu,[\xi,\eta]\rpair=0$ for all $\eta\in\g_\mu^\perp$ but if now $\eta\in \g_\mu$ then $0=\lpair \ad^*_\eta \mu,\xi\rpair=-\lpair \mu,[\xi,\eta]\rpair=\lpair \ad^*_{\xi} \mu,\eta\rpair$ for any $\eta$ in $\g$ but this implies that $\ad^*_\xi\mu=0$ and as $\xi\in \g_\mu^\perp$ then $\xi=0$. Denote by $\omega=\Omega^\mu \restr{\g_\mu^\perp}$ the restriction. The form $\omega$ is  symplectic  on $\g_\mu^\perp$.

Define now $\l:=\h\cap \g_\mu^\perp$ and \[\o=\{\lambda\in \g_\mu^\perp\cap {\h}^\perp\subset \g \mid \langle\ad^*_\lambda\mu,\eta\rangle=0\quad \forall \eta\in {\h}\}.\]
If $\xi\in \g_\mu^\perp$ is $\omega$-orthogonal to $\l$ then it must lie in $\o\oplus \l$ because $\xi$ can be decomposed as $\xi=\xi_1+\xi_2$ with $\xi_1\in \h\cap \g_\mu^\perp=\l$ and $\xi_2\in \h^\perp \cap \g_\mu^\perp$ but then as $\langle\mu,[\xi_2,\eta]\rangle=\langle\mu,[\xi,\eta] \rangle=0$ for any $\eta\in \h$
then $\xi_2\in \o$. That is, $\l^\omega \subset \o\oplus \l$. Conversely, if $\xi\in \o$ then by definition of $\o$ $\xi\in \l^\omega$, and if $\xi\in \h$ for any $\eta \in \l$ we have $\langle \mu,[\xi,\eta]\rangle=0$ because $\l\in\h$ and $\mu\in [\h,\h]^\circ$ so $\xi\in \l^\omega$ and, therefore, $\l^\omega = \o\oplus \l$.

Let $\xi\in \o\cap {\o}^\omega$.  Noting that $\xi \in \l^\omega$ we have
$\xi\in {\o}^\omega\cap \l^\omega=(\o\oplus \l)^\omega=(\l^\omega)^\omega=\l$
but as $\o\cap \l=0$ this implies that $\xi=0$, thus the restriction $\omega\restr{\o}$  is non-degenerate.

To build the space $\n$ we will need a preliminary standard result in linear algebra.
\begin{lem}
Let $A,B,C\subset E$ be three linear subspaces of a linear space $E$ such that $A\subset B$ and $A\cap C=0$. Then \[B\cap (C\oplus A)=(B\cap C)\oplus A\]
\end{lem}
Note that $\l\subset \o^\omega$ and as $\lpair \mu,[\xi,\eta]\rpair=0$ for any $\xi,\eta \in \l$, then $\l$ is an isotropic subset of the symplectic subspace $\o^\omega$, but in fact
\begin{align*}
\l^\omega\cap \o^\omega=\o^\omega\cap (\o\oplus \l)=(\o^\omega\cap\o)\oplus \l=\l
\end{align*}
where we applied the previous lemma with $A=\l$, $B=\o^\omega$ and $C=\o$. This implies that $\l$ is actually a Lagrangian subspace of $\o^\omega$ and it is clearly $H_\mu$-invariant. By Lemma 7.1.2 of \cite{ortegaratiu}
there must exists a $H_\mu$-invariant complement $\n\subset \g_\mu^\perp$ of $\l$. That is, we have $\g=\g_\mu\oplus \o\oplus \l\oplus \n$ and with respect to this splitting $\Omega^\mu$ block diagonalizes as
\begin{equation*}\Omega^\mu=\begin{bmatrix}
   0 & 0 & 0 & 0\\
   0 & \Omega\restr{\o} & 0 & 0 \\
   0 & 0 & 0 & * \\
   0 & 0 & * & 0
  \end{bmatrix}
\end{equation*}
where the entries $*$ will not be important in our discussion.
\end{proof}
\begin{rem}
The subspace $\o$ was introduced in \cite{MR2421706} by a different procedure. In that work $\o$ was constructed as a symplectic slice for the $H$-action on the coadjoint orbit $\Orb_\mu\subset \g^*$.  In fact the subspaces $\l$ and $\n$ can be understood as part of a Witt-Artin decomposition $T_\mu\Orb_\mu=\l\cdot \mu\oplus\n\cdot \mu\oplus \o\cdot \mu$ (see \cite{ortegaratiu}).
Note that as vector spaces supporting a $H_\mu$-action both $\l$ and $\n$ are isomorphic to the quotient $\h/\h_\mu$ and $\o$ is isomorphic to $\h^{\Omega_\mu}/(\Ker \Omega_\mu + \h)$.
\end{rem}

\subsection{The symplectic slice}

The first step towards the construction of a Hamiltonian tube for cotangent-lifted actions is to describe the symplectic slice at a point $z=(q,p)\in T^*Q$. This was done in  \cite{MR2293645} under the assumption $G_q\subset G_\mu$. Later, in \cite{MR2421706} the symplectic  slice for the general case was worked out.
Before stating the result we will introduce some new notation that will be used throughout the paper.  Let $V$ be a linear space supporting a linear representation of a group $G$ and let  $a\in V$ and $b\in V^*$. The \emph{diamond product} (see \cite{MR1627802}) $a \diamond \alpha \in \g^*$ is defined as \[\langle a\diamond\alpha,\eta\rangle:=\langle\alpha,\eta \cdot a\rangle \]
 for all $\eta\in \g$. With this notation the cotangent lift of a linear $G$-action on $V$ has momentum map $\J(a,b)=a\diamond b \in \g^*$. Note that if we consider $G=SO(3)$ acting on $\R^3$ the diamond product is just the usual cross product. If $\h\subset \g$ is a subspace then $a\diamond_\h b:=(a\diamond b)\restr{\h}\in \h^*$.
 If $\h$ is the Lie algebra of a subgroup $H\subset G$, $a\diamond_\h b$ is the momentum map for the $H$-action on $T^*V$ induced by  restriction of the original $G$-action, which is in turn the same as the lift of the restricted $H$-action on $V$. In this context, the next result gives a characterization of the symplectic slice for a cotangent-lifted action.

\begin{thm}[Theorem 6.1 of \cite{MR2421706}]
 Let $z=(q,p)\in T^*Q$ and $S$ be a $G_q$-invariant complement of $\g\cdot q$ in $T_qQ$.
 Define $H:=G_q$, $\mu:=\J(z)$, $\alpha:=z\rrestr{S}\in S^*$,  $B:=(\h_\mu\cdot \alpha)^\circ \subset S$. Let $\o$ be as defined in \eqref{eq:g-splitting}

The symplectic slice
$N$ at $z$ is linearly and $G_z$-equivariantly symplectomorphic to the product $\o\times T^*B$ equipped with the symplectic form
 \begin{equation}\Omega_N((\lambda_1,v_1,w_1)(\lambda_2,v_2,w_2))=-\langle \mu,[\lambda_1,\lambda_2]\rangle+\langle w_2,v_1\rangle -\langle w_1,v_2\rangle \label{eq:symplslice}\end{equation}
 the corresponding momentum map for the linear $G_z$-action is
 \[\J_N(\lambda,(a,\beta))=\frac{1}{2}\lambda\diamond_{\g_z}\ad^*_\lambda\mu+a\diamond_{\g_z} \beta.\]

\label{thm:symplnorm-cotang}
\end{thm}

\section{$G$-tubes}
\label{sec:gtubes}

In this section we will define both simple and restricted $G$-tubes. These maps will be the building blocks needed to find an explicit Hamiltonian tube for cotangent-lifted actions.

\subsection{Simple $G$-tubes}
From now on we will identify $TG$ with $G\times \g$ and $T^*G$ with $G\times \g^*$ using left trivializations
\begin{align*}
G\times \g & \longrightarrow TG & G\times \g^* & \longrightarrow T^*G\\
(g,\xi) &\longmapsto T_eL_g(\xi)   &(g,\nu) &\longmapsto T_e^*L_{g^{-1}}(\nu).
\end{align*}
 Combining them we can trivialize $T(T^*G)\cong G\times \g\times \g^*\times \g^*$.

We will need the following well-known properties of the symplectic structure and the cotangent-lifted actions of $G$ on $T^*G$ (see  \cite{MR515141})
\begin{prop}
Let $G$ act on itself by the left multiplication action and by cotangent lifts on $T^*G$ then we have
 \begin{itemize}
  \item \emph{Symplectic structure}: Let $u_i:=(\xi_i,\beta_i)\in T_{(g,\nu)}T^*G$ with $i=1,2$, the canonical one-form of $T^*G$ is \begin{equation}\theta_{T^*G}(u_1)=\lpair \nu,\xi_1\rpair\label{eq:oneform}\end{equation} and the canonical symplectic form $\omega_{T^*G}=-\ed\theta_{T^*G}$ is
  \begin{equation}
  \omega_{T^*G}(u_1,u_2)=\lpair\beta_2,\xi_1\rpair-\lpair\beta_1,\xi_2\rpair
  +\lpair \nu,[\xi_1,\xi_2]\rpair . \label{eq:twoform}
 \end{equation}
  \item \emph{Cotangent-lifted left multiplication}: The $G$-action given by $h\cdot ^L(g,\nu)=(hg,\nu)$ has as infinitesimal generator $\eta_{T^*G}^L(g,\nu)=(\Ad_{g^{-1}}\eta,0)$ and  is Hamiltonian with momentum map $\J_L(g,\nu)=\Ad_{g^{-1}}^* \nu$.

  \item \emph{Cotangent-lifted right multiplication}: The $G$-action given by $h\cdot ^R(g,\nu)=(gh^{-1},\Ad_{h^{-1}}^*\nu)$ has as infinitesimal generator $\eta_{T^*G}^R(g,\nu)=(-\eta,-\ad^*_\eta \nu)$ and  is Hamiltonian with momentum map $\J_R(g,\nu)=-\nu$.
 \end{itemize}
\label{prop:T*G}
\end{prop}
Notice that if we think of  $\g^*$ as a manifold endowed with the $G$-action $g\cdot \nu=\Ad^*_{g^{-1}}\nu$ then the above left and right actions of $G$ on $T^*G$ are exactly the actions $G^L$ and $G^T$ on the product manifold $G\times \g^*$ (see Subsection \ref{sec:twisted}).

We  now define simple $G$-tubes.
\begin{defn}
 Let $H$ be a compact subgroup of $G$ and $\mu\in \g^*$. Given a splitting
 $\g=\g_\mu\oplus \q$ invariant under the $H_\mu$-action, a simple $G$-tube is a map
 \begin{equation*} \Theta:G\times U \subset G\times (\g_\mu^*\times \q)\longrightarrow G\times \g^*\cong T^*G 
 \end{equation*}
 such that:
 \begin{enumerate}
  \item $U$ is a connected $H_\mu$-invariant neighborhood of 0 in $\g_\mu^*\times \q$.
  \item $\Theta$ is a $G^L$-equivariant diffeomorphism onto $\Phi(G\times U)$ satisfying $\Theta(e,0)=(e,\mu)$.
  \item Let $u_i:=(\xi_i,\dot\nu_i,\dot\lambda_i)\in T_{(g,\nu,\lambda)} G\times \g_\mu^*\times \q $ with $i=1,2$, then
  \begin{equation}\begin{split}
  (\Theta^*\omega_{T^*G})(u_1,u_2)=
  \lpair \dot\nu_2,\xi_1\rpair
  -\lpair \dot \nu_1, \xi_2\rpair
  +\lpair \nu+\mu,[\xi_1, \xi_2]\rpair-\lpair\mu,[\dot\lambda_1,\dot\lambda_2]\rpair
 \end{split} \label{eq:simpletube2form}\end{equation}

  \item $\Theta$ is $H_\mu^T$-equivariant.
  \item $T_{(e,0,0)}\Theta(\xi,\dot \nu,\dot \lambda)=(\xi+\dot \lambda;\dot \nu+\ad^*_{\dot \lambda} \mu)\in T_{(e,0)}(T^*G)$.
 \end{enumerate}
\label{def:simpletube}
\end{defn}
If $\q$ is defined as above, note that the symplectic slice for the cotangent-lifted left multiplication of $G$ on $T^*G$ at $(e,\mu)\in T^*G$ is precisely $\q$. Indeed, as $T_{(e,\mu)} \J_L(e,\mu)\cdot (\xi,\dot \nu)=-\ad^*_{\xi}\mu+\dot \nu$ then a complement to $\g_\mu\cdot(e,\mu)$ can be chosen to be the space $\{(\xi,\ad^*_{\xi}\mu) \mid \xi \in \q\}$, and using \eqref{eq:twoform}, this linear space is symplectomorphic to $(\q,\Omega^\mu\rrestr{\q})$.
According to Theorem \ref{thm:classicalMGS}, the MGS model at $(e,\mu)\in T^*G$ for the free cotangent-lifted left multiplication of $G$ on $T^*G$ will be of the form $G\times \g_\mu^*\times \q$ and, in this case, the symplectic form \eqref{eq:MGSsymp} is precisely the one given by\eqref{eq:simpletube2form}. In other words, a simple $G$-tube is a Hamiltonian tube for $T^*G$ at $(e,\mu)$ (properties 1--3) but we further require $H_\mu^T$-equivariance and a prescribed property on its linearization (properties 4--5).

The next result ensures the existence of simple $G$-tubes. The idea is that an equivariant version of the Moser trick can be used to construct them. This part follows closely Theorem 6 in \cite{MR1486529} and Theorem 7.3.1 in \cite{ortegaratiu}. We are going to apply Moser's trick  to an explicit, well-behaved, family of symplectic potentials.
\begin{prop}[Existence of simple $G$-tubes]
Given a $H_\mu$-invariant splitting $\g=\g_\mu\oplus \q$ there exists a $H_\mu$-invariant open neighborhood $D$ of $0\in \g_\mu^*\times \q$ and a simple $G$-tube  $\Theta:G\times D\subset G\times \g_\mu^* \times \q \to G\times \g^*$.
\label{prop:simpleGtube}
\end{prop}

\begin{proof}
As a first approximation we will consider the map
\begin{align}
F: G\times \g_\mu^*\times \q & \longrightarrow G\times \g^* \label{eq:defF}\\
(g,\nu,\lambda) & \longmapsto (g \exp(\lambda),\Ad_{ \exp(\lambda)}^*(\nu+\mu)) \nonumber
\end{align}
defined only for $\lambda$ small enough so that it is contained in the injectivity domanin of the group exponential $\exp:\g\to G$. The map $F$ is $G^L$-equivariant and also $H_\mu^T$-equivariant because $$F(g'g,\nu,\lambda)=(g'g\exp(\lambda),\Ad_{\exp(\lambda)}(\nu+\mu)),\quad\text{and}$$
\begin{eqnarray*}F(gh^{-1},\Ad^*_{h^{-1}}\nu,\Ad_h\lambda) & = & (gh^{-1}\exp(\Ad_h \lambda),\Ad_{\exp(\Ad_h \lambda)}^* (\Ad^*_{h^{-1}}\nu+\mu)\\ & = & (g\exp(\lambda)h^{-1},\Ad_{h^{-1}}^*\Ad_{ \exp(\lambda)}^*(\nu+\mu)). \end{eqnarray*}
Consider now the one-form on $G\times (\g_\mu^*\times \q)$ given by $\theta_Y(g,\nu,\lambda)(\xi,\dot \nu,\dot \lambda)=\lpair \nu+\mu,\xi\rpair+\frac{1}{2}\lpair \mu,\ad_\lambda \dot\lambda\rpair+\lpair \mu,\dot\lambda\rpair$. It is clearly $G^L$-invariant and  $H_\mu^T$-invariant because
\[\begin{split}
   \theta_Y(gh^{-1},\Ad^*_{h^{-1}}\nu,\Ad_h\lambda)(\Ad_h \xi,\Ad^*_{h^{-1}} \dot \nu,\Ad_h \dot \lambda)= \\
   =\lpair\Ad^*_{h^{-1}}(\nu)+\mu,\Ad_h\xi\rpair+\frac{1}{2}\lpair \mu,[\Ad_h\lambda, \Ad_h\dot\lambda]\rpair+\lpair \mu,\Ad_h\dot\lambda\rpair= \theta_Y(g,\nu,\lambda)(\xi,\dot \nu,\dot \lambda).
  \end{split}
\]
Let $u_i:=(\xi_i,\dot\nu_i,\dot\lambda_i)\in T_{(g,\nu,\lambda)} (G\times \g_\mu^*\times \q )$ with $i=1,2$ note that $(-\ed \theta_Y)(u_1,u_2)$ is the right-hand side of equation \eqref{eq:simpletube2form}.
Consider now the family of $G^L\times H_\mu^T$-invariant one-forms
$$\theta_t=tF^*\theta_{T^*G}+(1-t)\theta_Y$$ and define $\omega_t:=-\ed \theta_t$. Using \eqref{eq:twoform} and $T_{(e,0,0)} F(\xi,\dot \nu,\dot \lambda)=(\xi+\dot \lambda,\dot \nu+\ad^*_{\dot \lambda} \mu)$ it can be checked that
\[(-\ed\theta_t)(g,0,0)(\xi_1,\dot \nu_1,\dot \lambda_1)(\xi_2,\dot \nu_2,\dot \lambda_2)=\lpair \dot \nu_2,\xi_1\rpair-\lpair \dot \nu_1,\xi_2\rpair+\lpair \mu,[\xi_1,\xi_2]\rpair-\lpair \mu,[\dot \lambda_1,\dot \lambda_2]\rpair,\]
but this two-form is non-degenerate because it corresponds precisely to  $\Omega_Y$ of Theorem \ref{thm:classicalMGS}. This implies that Moser's equation $i_{X_t}\omega_t=\frac{\partial \theta_t}{\partial t}$ defines a time-dependent vector field $X_t$ on an open set $G\times V\subset G\times \g_\mu^*\times \q$. If $\Psi_t$ is the local flow of  $X_t$ then $\Psi_t^*\omega_t=\omega_0$ (see Theorem 7.3.1 of \cite{ortegaratiu} for technical details).
As $\theta_t$ and $-\ed\theta_t$ are $G^L\times H_\mu^T$ invariant differential forms, then the vector field $X_t$ is $G^L\times H_\mu^T$ invariant and, therefore, the local flow $\Psi_t$ is $G^L\times H_\mu^T$-equivariant for any $t$. Note that $\theta_Y(g,0,0)=\lpair \mu,\xi \rpair+\lpair \mu,\dot \lambda\rpair$ and $F^*\theta_{T^*G}(g,0,0)=\lpair \mu,\xi\rpair + \lpair \mu,\dot \lambda\rpair$. This implies that $\frac{\partial\theta_t}{\partial t}\rrestr{(g,0,0)}=0$ and $X_t(g,0,0)=0$ so $\Psi_t(g,0,0)=(g,0,0)$ for any $t\in \R$ and then there is a $H_\mu$-invariant open set $U\subset V$  such that $\Psi_1$ is a diffeomorphism with domain $G\times U$.

The simple $G$-tube will then be the composition $\Theta=F\circ \Psi_1:G\times U\longrightarrow T^*G$. It is $G^L\times H_\mu^T$-equivariant and it satisfies $\omega_Y=\omega_0=\Psi_1^*\omega_1=\Psi_1^*F^*\omega_{T^*G}=\Theta^*\omega_{T^*G}$ and $\Theta(e,0)=(e,\mu)$.
Let $\Psi_t$ be the local flow of $X_t$ and $\eta_t$ be any time-dependent tensor field then
\begin{equation}\frac{d}{dt}\Psi_t^*\eta_t=\Psi_t^*\left(\mathcal{L}_{X_t}\eta_t+\frac{d}{dt}\eta_t\right).\label{eq:dt_time_flow}\end{equation}
This expression can be used to compute $T_{(e,0,0)} \Theta$. To do so
let $Y$ be any time-independent vector field on $G\times \g_\mu^* \times \q$ not vanishing at $(e,0,0)$. As $X_t$ vanishes at $(e,0,0)$ then $\mathcal{L}_{X_t} Y \restr{(e,0,0)}=0$. Setting $\eta_t=Y$ in \eqref{eq:dt_time_flow} it gives $\frac{d}{dt}\Psi_t^* Y=0$ but this implies $T_{(e,0,0)}\Psi_1=\text{Id}$ and, therefore, $T_{(e,0,0)}\Theta (\xi,\dot \nu,\dot \lambda)=(\xi+\dot \lambda,\dot \nu+\ad^*_{\dot \lambda} \mu)$. That is, $\Theta$ satisfies all the five required conditions for a simple $G$-tube.
\end{proof}

The main shortcoming with the previous existence result is that, as happens with Theorem \ref{thm:classicalMGS}, it does not
 produce an explicit map and it relies on the integration of a time-dependent field. However, we will see in Section \ref{sec:explicitexamples} that in some particular cases we can explicitly describe these objects. Nevertheless, using momentum maps we can still find a simpler expression for the simple $G$-tube $\Theta$. Decompose $\Theta$ as $$\Theta(g,\nu,\lambda)=(A(g,\nu,\lambda),B(g,\nu,\lambda))\in G\times \g^*.$$ The property of $G^L$-equivariance
 implies that $A(g,\nu,\lambda)=gA(e,\nu,\lambda)$. As $\Theta(e,0,0)=(e,\mu)$ then $A(e,0,0)=e$ and  $B(e,0,0)=\mu$.

Using Section \ref{sec:reduced-MGS} we have that the product $G\times \g_\mu^*\times \q$ is equipped with $G^L$ and $H_\mu^T$ Hamiltonian actions with momentum maps $\mathbf{K}_{G^L}$ and $\mathbf{K}_{H_\mu^T}$  respectively (see \eqref{eq:Zmomentummaps}). We also have $G^L$ and $H_\mu^T$ Hamiltonian actions on $G\times\g^*$ and their momentum maps are $\J_{G^L}$ and $\J_{H_\mu^T}$ (see Proposition \ref{prop:T*G}). As the difference between two momentum maps is a locally constant function and both $\J_{G^L}$ and $\mathbf{K}_{G^L}$ are equivariant then$\J_{G^L}\circ \Theta=\mathbf{K}_{G^L}$, that is
\[\Ad^*_{A(g,\nu,\lambda)^{-1}}B(g,\nu,\lambda)=\Ad^*_{g^{-1}}(\nu+\mu)\]
so $B(g,\nu,\lambda)=\Ad^*_{A(g,\nu,\lambda)}\Ad^*_{g^{-1}}(\nu+\mu)=\Ad^*_{g^{-1}A(g,\nu,\lambda)}(\nu+\mu)=\Ad^*_{A(e,\nu,\lambda)}(\nu+\mu)$.
If we denote $E(\nu,\lambda)=A(e,\nu,\lambda)$ then we can write
\begin{align}
\Theta:G\times \g_\mu\times \q &\longrightarrow T^*G \nonumber\\
(g,\nu,\lambda)& \longmapsto (gE(\nu,\lambda),\Ad^*_{E(\nu,\lambda)}(\nu+\mu)). \label{eq:simpleGtube-E}
\end{align}
Therefore, a simple $G$-tube is determined by a function $E:U\subset \g_\mu^*\times \q\to G$. Rewriting Definition \ref{def:simpletube} in terms of the function $E$ we could obtain necessary and sufficient conditions for $E$.

\begin{rem}
Note that if $\g=\g_\mu$, which is the hypothesis used in \cite{MR2293645}, then $\q=0$ and the shifting map
 \begin{align*}
  G\times \g^* &\longrightarrow G\times \g^*\\
  (g,\nu) &\longmapsto (g,\nu+\mu)
 \end{align*}
is a simple $G$-tube.
\label{rem:isotropicmu}
\end{rem}

\begin{rem}
As $\Theta$ is $H_\mu^T$-equivariant, the momentum preservation argument gives
\begin{equation}\J_{H_\mu^T}(\Theta(g,\nu,\lambda))=\mathbf{K}_{H_\mu^T}(g,\nu,\lambda)=-\nu\rrestr{\h_\mu}+\frac{1}{2}\lambda\diamond_{\h_\mu}\ad^*_\lambda\mu. \label{eq:simpletubeH_mu_momentum}\end{equation}
That is, we have the condition $(\Ad^*_{E(\nu,\lambda)}(\nu+\mu))\rrestr{\h_\mu}=\nu\rrestr{\h_\mu}-\frac{1}{2}\lambda\diamond_{\h_\mu}\ad^*_\lambda\mu$. This property will be  useful later during the proof of Proposition \ref{prop:fiberedBL}.
\end{rem}

\subsection{Restricted $G$-tubes}

If $G$ acts freely on $Q$ we will see in Section \ref{sec:general-tube} that the simple $G$-tube is enough to construct explicitly the Hamiltonian tube for $T^*Q$ but for non-free actions we will need to adapt a simple $G$-tube to the corresponding isotropy subgroup, the result being the restricted $G$-tube.
 \begin{defn}
Given an adapted splitting $\g=\g_\mu\oplus \o\oplus\l\oplus \n$ as in Proposition \ref{prop:liealgsplitting}
 a restricted $G$-tube is a map
 \begin{equation*} \Phi:G\times U \subset G\times \g_\mu^*\times \o\times \l^* \longrightarrow T^*G 
 \end{equation*} such that:
 \begin{enumerate}
  \item $U$ is a connected $H_\mu$-invariant neighborhood of 0 in $\g_\mu^*\times \o\times \l^*$.
  \item $\Phi$ is a $G^L\times H_\mu^T$-equivariant diffeomorphism between $G\times U$ and $\Phi(G\times U)$ such that  $\Phi(e,0,0;0)=(e,\mu)$.
  \item Let $u_i:=(\xi_i,\dot\nu_i,\dot\lambda_i,\dot \eps_i)\in T_{(g,\nu,\lambda,\eps)} G\times \g_\mu^*\times \o\times \l^* $ with $i=1,2$, then $\Phi^*\omega_{T^*G}$ is
  \begin{equation}\begin{split}
  \omega_\text{restr}(u_1,u_2)=
  \lpair \dot \nu_2,\xi_1\rpair
  -\lpair \dot \nu_1,\xi_2\rpair
  +\lpair \nu+\mu,[\xi_1,\xi_2]\rpair-\lpair\mu,[\dot\lambda_1,\dot\lambda_2]\rpair
 \end{split}.\label{eq:restr2form}\end{equation}
  \item $\J_R(\Phi(g,\nu,\lambda,\eps))\restr{\l}=-\eps$ for any $(g,\nu,\lambda,\eps)$ where $\J_R$ is the momentum map for the $G^R$-action on $T^*G$ (see Proposition \ref{prop:T*G}).
 \end{enumerate}
\label{def:restrictedtube}
\end{defn}

If we are given a simple $G$-tube $\Theta$ then we can build a restricted $G$-tube $\Phi$ solving a non-linear equation. In fact, the restricted $G$-tube will be of the form $\Phi(g,\nu,\lambda,\eps)=\Theta(g,\nu,\lambda+\zeta(\nu,\lambda,\eps))$ for some map $\zeta:\g_\mu^*\times \o\times \l^*\to \n$. This is the main idea behind the following result.

\begin{prop}[Existence of restricted $G$-tubes]
Given an adapted splitting $\g=\g_\mu\oplus \o\oplus\l\oplus \n$ as in Proposition \ref{prop:liealgsplitting}
there is an $H_\mu$-invariant open neighborhood $D$ of $0\in \g_\mu^*\times \o\times \l^*$ and a restricted $G$-tube $\Phi:G\times D\rightarrow T^*G$.
\label{prop:restrtubes}
\end{prop}

\begin{proof}

Define $\q=\o\oplus\l \oplus \n$. Using Proposition \ref{prop:simpleGtube} there exists a simple $G$-tube $\Theta$ defined on the symplectic space $Y:=G\times U\subset G\times (\g_\mu^*\times \q)$ with symplectic form $\omega_Y$ \eqref{eq:simpletube2form}. As $U\subset \g_\mu^*\times \q$ is a neighborhood of $0$ there are $H_\mu$-invariant neighborhoods of the origin $(\g_\mu^*)_r\subset \g_\mu^*$, $\o_r\subset \o$ and $\n_r\subset \n$ such that $(\g_\mu^*)_r\times(\o_r+\n_r)\subset U$.
Consider now the map
\begin{align*}
\iota_W:W=G\times ((\g_\mu^*)_r\times \o_r\times \n_r)& \longrightarrow Y=G\times U\subset G\times (\g_\mu^*\times \q) \\
(g,\nu,\lambda,\zeta) &\longmapsto (g,\nu,\lambda+\zeta)
\end{align*}
This map is a $G^L\times H_\mu^T$-equivariant  embedding.
By the properties of the adapted splitting (see Proposition \ref{prop:liealgsplitting}) $\Omega^\mu(\lambda,\zeta)=0$ if $\lambda\in \o$ and $\zeta\in\n$, therefore,
  \[\begin{split}
  (\iota_W^*\omega_Y)(u_1,u_2)=
  \lpair \dot \nu_2,\xi_1\rpair
  -\lpair \dot \nu_1,\xi_2\rpair
  +\lpair \nu+\mu,[\xi_1,\xi_2]\rpair-\lpair\mu,[\dot\lambda_1,\dot\lambda_2]\rpair
 \end{split}\]
where $u_i:=(\xi_i,\dot\nu_i,\dot\lambda_i,\dot \zeta_i)\in T_{(g,\nu,\lambda,\zeta)} G\times \g_\mu^*\times \o\times \n$ with $i=1,2$.
In order to obtain the restricted $G$-tube we will need to impose the relationship between $\eps$ and $\J_R$. To do so define the map
\begin{align*}
\psi: W &\longrightarrow G\times \g_\mu^*\times \o \times \l^* \\
(g,\nu,\lambda,\zeta) &\longmapsto (g,\nu,\lambda;-\J_R(\Theta(g,\nu,\lambda+\zeta))\rrestr{\l}).
\end{align*}
Note that this map is $G^L\times H_\mu^T$-equivariant because $$(g'g,\nu,\lambda;-\J_R(\Theta(g'g,\nu,\lambda+\zeta))\rrestr{\l})= (g'g,\nu,\lambda;-\J_R(g'\Theta(g,\nu,\lambda+\zeta))\rrestr{\l})=(g'g,\nu,\lambda;-\J_R(\Theta(g,\nu,\lambda+\zeta))\rrestr{\l})$$ and \begin{align*}
\psi(h\cdot^T (g,\nu,\lambda,\zeta)) &=\left(gh^{-1},\Ad^*_{h^{-1}}\nu,\Ad_h\lambda;-\J_R(\Theta(h\cdot^T(g,\nu,\lambda+\zeta))\rrestr{\l}\right)\\
&=\left(gh^{-1},\Ad^*_{h^{-1}}\nu,\Ad_h\lambda;-\J_R(h\cdot^T\Theta((g,\nu,\lambda+\zeta))\rrestr{\l}\right)\\
&=\left(gh^{-1},\Ad^*_{h^{-1}}\nu,\Ad_h\lambda;-\Ad^*_{h^{-1}}\J_R(\Theta((g,\nu,\lambda+\zeta))\rrestr{\l}\right)\\
&=h\cdot^T\left(g,\nu,\lambda;-\J_R(\Theta(g,\nu,\lambda+\zeta))\rrestr{\l}\right).
\end{align*}
Moreover, if we endow $G\times \g_\mu^*\times \o \times \l^*$ with the two-form \eqref{eq:restr2form} then $\psi^*\omega_\text{restr}=\iota_W^*\omega_Y$. We will now check that $\Psi$ is invertible. Let $v:=(\xi,\dot \nu,\dot \lambda,\dot \zeta)\in T_{(e,0,0,0)}G\times \g_\mu^*\times \o \times \n$, then
\begin{align*}
 T_{(e,0,0,0)}\psi \cdot v &= \left(\xi,\dot\nu,\dot\lambda;-T_{(e,0,0)}\left(\J_R\rrestr{\l}\circ \Theta\right)\cdot(\xi,\dot\nu,\dot\lambda+\dot\zeta)\right) \\
 &= \left(\xi,\dot\nu,\dot\lambda;-T_{(e,0)}\left(\J_R\rrestr{\l}\right)\cdot(\xi+\dot\lambda+\dot\zeta,\dot\nu+\ad^*_{\dot\lambda+\dot\zeta}\mu)\right) \\
 &= \left(\xi,\dot\nu,\dot\lambda;-(\dot\nu+\ad^*_{\dot\lambda+\dot\zeta}\mu)\rrestr{\l}\right)\\
 &= \left(\xi,\dot\nu,\dot\lambda;-(\ad^*_{\dot\zeta}\mu)\rrestr{\l}\right)
\end{align*}
where we have used the expression for $T_{(e,0,0)}\Theta$ given in Definition \ref{def:simpletube} and that $\ad^*_{\dot\lambda}\mu\rrestr{\l}=0$ since $\o$ and $\l$ are $\Omega^\mu$-orthogonal (see Proposition \ref{prop:liealgsplitting}).

  Since $\l$ and $\n$ are complementary isotropic subspaces, the map $\sigma:\n\to \l^*$ given by $\sigma(\zeta)= \ad^*_\zeta \mu \restr{\l}$ is a linear $H_\mu$-equivariant isomorphism.
This implies that $T_{(e,0,0,0)}\psi$ is invertible. By the Inverse Function Theorem
there is a neighborhood of $(e,0,0,0)\in G\times \g_\mu^*\times \o \times \l^*$ on which $\psi^{-1}$ is well defined. Due to $G^L\times H_\mu^T$ equivariance of $\Psi$
this neighborhood must be of the form $G\times D$ with $D\subset \g_\mu^*\times \o \times \l^*$ a $H_\mu$-invariant neighborhood of zero.

Note that the composition $\Theta\circ\iota_W\circ \psi^{-1}$ is a restricted $G$-tube because it satisfies $$(\Theta\circ\iota_W\circ \psi^{-1})^*\omega_{T^*G}=(\iota_W\circ\psi^{-1})^*\omega_Y=\omega_\text{restr},$$ it is $G^L\times H_\mu^T$-equivariant (because it is the composition of $G^L\times H_\mu^T$-equivariant maps), the origin $(e,0,0,0)$ is mapped to $(e,\mu)\in T^*G$, and it is a diffeomorphism onto its image (because it is a composition of diffeomorphisms onto its images). Finally, if $(g,\nu,\lambda,\eps)=\psi(g,\nu,\lambda,\zeta)$ then $\J_R(\Theta(g,\nu,\lambda+\zeta))\rrestr{\l}=-\eps$, that is $(\J_R\rrestr{\l}\circ \Theta\circ\iota_W\circ \psi^{-1})(g,\nu,\lambda,\eps)=-\eps$, which is the condition needed for a restricted $G$-tube.

To sum up, the composition $\Phi:=\Theta\circ\iota_W\circ \psi^{-1}:G\times D \to T^*G$ is a restricted $G$-tube. This map can also be written as
\begin{align}
\Phi: G\times D\subset G\times \g_\mu^*\times \o \times \l^* &\longrightarrow T^*G \label{eq:implicitPhi}\\
(g,\nu,\lambda;\eps) &\longmapsto \Theta(g,\nu,\lambda+ \zeta(\nu,\lambda;\eps)) \nonumber
\end{align}
where $\zeta:D\subset \g_\mu^*\times \o\times \l^*\to \n$ is determined by the equation $\J_R\rrestr{\l}(\Phi(g,\nu,\lambda+\zeta,\eps))=-\eps$.

\end{proof}

\section{Cotangent bundle Hamiltonian tubes}
\label{sec:cotangenttubes}

Let $G$ be a Lie group acting properly on $Q$, and fix $z\in T^*Q$. In this section we will construct a Hamiltonian tube for the cotangent-lifted action of $G$ on $T^*Q$ around $z$ that will be explicit except for the computation of a restricted $G$-tube. This Hamiltonian tube will be a generalization of the construction  in \cite{MR2293645} under the hypothesis $G_\mu=G$.

\subsection{Cotangent-lifted Palais model}
We will first reduce the problem on $T^*Q$ to a problem on $T^*(G\times_H S)$. This first simplification is already discussed in \cite{MR2293645}.
Recall the well-known regular reduction theorem for cotangent bundles at zero momentum \cite{MR0448428}.

\begin{thm}[Regular cotangent reduction at zero] Let $G$ act freely and properly by cotangent lifts on $T^*Q$ with momentum map $\J$. Denote by $\pi_G$ the projection $Q\to Q/G$ and consider the map $\varphi:\J^{-1}(0)\to T^*(Q/G)$ defined by $\lpair \varphi(z), T\pi_G(v)\rpair=\lpair z,v\rpair$ for every $z\in T^*_qQ$ and $v\in T_qQ$. Let $\pi_0$ and $\iota$ by the natural projection $\pi_0:\J^{-1}(0)\to \J^{-1}(0)/G$ and the inclusion  $\iota:\J^{-1}(0)\to T^*Q$.
The map $\varphi$ is a $G$-invariant surjective submersion that induces a symplectomorphism
\[\bar\varphi:\J^{-1}(0)/G\longrightarrow T^*(Q/G)\]
where $\J^{-1}(0)/G$ is endowed with the unique symplectic form $\omega_0$ satisfying and $\pi_0^*\omega_0=\iota^*\omega_{T^*Q}$.
\label{thm:zerocotangentred}
\end{thm}

Let $G$ be a Lie group and $H$ a compact subgroup that acts linearly on the linear space $S$. Let $S_r$ be a $H$-invariant open neighborhood of $0\in S$. We will consider the symplectic space $T^*(G\times S_r)$ that can be identified with $G\times \g^*\times S_r\times S^*$ using the left-trivialization of $G$ and the linear structure of $S$. In Section \ref{sec:twisted} we introduced the $G^L$ and $H^T$ actions on the space $G\times S_r$. These actions can be lifted to Hamiltonian actions on $T^*(G\times S_r)$. More explicitly, using Proposition \ref{prop:T*G} and the diamond notation we have
\begin{itemize}
 \item cotangent-lifted $G^L$-action: $g'\cdot^L (g,\nu,a,b)=(g'g,\nu,a,b)$ with momentum map $$\J_{G^L}(g,\nu,a,b)=\Ad^*_{g^{-1}}\nu.$$
 \item cotangent-lifted $H^T$-action: $h\cdot^T (g,\nu,a,b)=(gh^{-1},\Ad^*_{h^{-1}} \nu, h\cdot a,h\cdot b)$ with momentum map $$\J_{H^T}(g,\nu,a,b)=-\nu\rrestr{\h}+a\diamond b.$$
\end{itemize}

Theorem \ref{thm:zerocotangentred} applied to $G\times S_r$ with the $H^T$-action gives the diagram
\begin{equation}\xymatrix{
 \J^{-1}_{H^T}(0) \ar@{^{(}->}[r] \ar[dr]^\varphi \ar[d] & T^*(G\times S_r)  \\
 \J^{-1}_{H^T}(0)/H^T \ar[r]^{\bar\varphi} & T^*(G\times_H S_r)
 } \label{eq:varphi-Z}
\end{equation}
but, since $\J_{H^T}^{-1}(0)$ is $G^L$-invariant the quotient $\J_{H^T}^{-1}(0)/H^T$ supports a $G$-action that it can be checked to be Hamiltonian with momentum map $\J_{\text{red}}([g,\nu,a,b]_{H^T})=\Ad_{g^{-1}}^* \nu$. That is, $\bar\varphi$ is in fact a $G$-equivariant symplectomorphism and $\varphi$ is a $G^L$-equivariant surjective submersion.

Consider now a Lie group $G$ acting properly on a general manifold $Q$ and by cotangent lifts on $T^*Q$ with momentum map $\J_{T^*Q}$. Our goal is to construct a Hamiltonian tube around an arbitrary point $z\in T^*Q$. For that, let $q=\tau(z)$ where $\tau:T^*Q\to Q$ is the projection, define $H:=G_q$ and consider a linear slice  at $q$, that is a $H$-invariant complement $S$ to $\g\cdot q$ in $T_qQ$. Using Palais' model \eqref{eq:Palais} there is a $H$-invariant neighborhood $S_r\subset S$ and a $G$-equivariant diffeomorphism
\(\mathbf{s}:G\times_{H} S_r\longrightarrow U\subset T^*Q\) satisfying $\mathbf{s}([e,0]_{H})=q$. As $\mathbf{s}$ is a diffeomorphism the cotangent lift $T^*\mathbf{s}^{-1}:T^*(G\times_H S_r)\longrightarrow \pi^{-1}(U)\subset Q$ 
is a $G$-equivariant symplectomorphism onto $T^*U\cong \tau^{-1}(U)\subset T^*Q$.

If we denote  $\alpha=z\rrestr{S}$ and $\mu=\J_{T^*Q}(z)$, then $T^*\mathbf{s}^{-1}(\varphi(e,\mu,0,\alpha))=z$ since $\tau(\varphi(e,\mu,0,\alpha))=[e,0]_H=\mathbf{s}^{-1}(q)$,
and as any $v\in T_qQ$ can be decomposed as $v=\xi\cdot q+\dot a$ with $\xi\in \g$
and $\dot a\in S$ then

\begin{align*}
 \langle T_{q}^*\mathbf{s}^{-1}(\varphi(e,\mu,0,\alpha)),v\rangle &= \langle \varphi(e,\mu,0,\alpha),T_q \mathbf{s}^{-1}\cdot v\rangle= \langle \varphi(e,\mu,0,\alpha),(\xi,\dot a)\rangle= \\
 &=\langle(\mu,\alpha),(\xi,\dot a)\rangle=\langle \mu,\xi\rangle + \langle \alpha,\dot a\rangle= \\
 &=\langle \J_{T^*Q}(z),\xi\rangle +\langle z,\dot a\rangle = \langle z,\xi\cdot q\rangle + \langle z,\dot a\rangle =\langle z,v\rangle.
\end{align*}

Therefore, from now on we will assume without loss of generality  $Q=G\times_H S_r$ and  $z=\varphi([e,\mu,0,\alpha]_{H})$ with $\mu\in \g^*$ and $\alpha\in S^*$. Moreover, this simplification is explicit up to the exponential of a metric.

The $G$-isotropy of $z$ is $G_z=H_\alpha \cap G_\mu$ since $G_z=G_{[e,\mu,0,\alpha]_H}=H_{(\mu,\alpha)}=H_\mu\cap H_\alpha=G_\mu \cap H_\alpha$. If $\p$ is a $H_\mu$-invariant complement of $\h_\mu$ in $\g_\mu$, $\s$ is a complement of $\g_z$ in $\h_\mu$ and $B=(\h_\mu\cdot \alpha)^\circ\subset S$ then, using Theorems \ref{thm:symplnorm-cotang} and \ref{thm:classicalMGS}, and the adapted splitting of Proposition \ref{prop:liealgsplitting}, the Hamiltonian tube around $z$ has to be a map of the form
\begin{equation}\label{eq:Talpha}\mathcal{T}:G\times_{G_z}\big(\underbrace{(\s^*\oplus \p^*)}_{\m^*}\times \underbrace{\o\times B\times B^*}_N\big)\longrightarrow T^*(G\times_H S).\end{equation}
The first difficulty that we find is that the MGS model is a $G_z$-quotient but the target space is a $H$-quotient. For this reason, instead of constructing directly the tube we are going to split it as the composition of two maps: one that goes from an $H_\mu$-quotient to an $H$-quotient (essentially this is done by a restricted $G$-tube)
 and another that goes from a $G_z$-quotient to an $H_\mu$-quotient. We will carefully explain this process in the following sections.

\subsection{The $\alpha=0$ case}

In this section we will construct a Hamiltonian around  a point of the form $z_0=\varphi(e,\mu,0,0)\in T^*(G\times_H S)$ which is explicit up to a restricted $G$-tube. Notice that $G_{z_0}=G_\mu\cap H_\alpha=H_\mu$ and by Theorem \ref{thm:symplnorm-cotang} and the adapted splitting of Proposition \ref{prop:liealgsplitting} the symplectic slice is $N_0=\o\times S\times S^*$ with symplectic form \eqref{eq:symplslice}. Then the map \eqref{eq:Talpha} reduces in this case to
\[\mathcal{T}_0:G\times_{H_\mu}(\underbrace{\p^*}_{\m^*}\times \underbrace{\o\times S\times S^*}_{N_0})\longrightarrow T^*(G\times_H S)\]
where $G\times_{H_\mu}(\p^*\times \o\times S\times S^*)$ is equipped with the symplectic form \eqref{eq:MGSsymp}.
As the general Hamiltonian tube will eventually factor through   $\mathcal{T}_0$ we will need to ensure that its domain is large enough, and in particular it should contain all the points of the form $[e,0,0,0,b]_{H_\mu}$ for all $b\in S^*$.

\begin{thm}\label{thm:tube0}
Consider the point $z_0=\varphi(e,\mu,0,0)\in T^*(G\times_H S)$. Let $\g=\g_\mu\oplus \o\oplus\l\oplus \n$ be an adapted splitting in the sense of Proposition \ref{prop:liealgsplitting} and let $\Phi:G\times U_\Phi \longrightarrow T^*G$ be an associated restricted $G$-tube.
In this context, there are $H_\mu$-invariant open neighborhoods of zero: $\p^*_r\subset \p^*$, $\o_r\subset \o$ and an $H$-invariant open neighborhood of zero $\h_r^*\subset \h^*$ such that the map
\begin{align}
\mathcal{T}_0:  G\times_{H_\mu}\left(\p^*_r\times \o_r\times (T^*S)_r\right) &\longrightarrow T^*(G\times_H S) \label{eq:alpha0tube}\\
[g,\nu,\lambda;a,b]_{H_\mu} &\longmapsto \varphi(\Phi(g,\tilde \nu,\lambda;a\diamond_\l b);a,b) \nonumber
\end{align}
is a Hamiltonian tube around the point $z_0$, where \[\tilde \nu=\nu+\underbrace{\frac{1}{2}\lambda \diamond_{\h_\mu} \ad_\lambda^*\mu+a\diamond_{\h_\mu} b}_{\J_{N_0}(\lambda,a,b)}\] and $(T^*S)_r:=\{(a,b)\in T^*S \mid a\diamond_\h b\in \h_r^*\}$.
\end{thm}

\begin{proof}

If we assume the existence of $(\p^*)_r,\o_r$ and $\h_r$ such that the map $\mathcal{T}_0$ is well defined then it follows from the properties of $\Phi$ that
\[\mathcal{T}_0([e,0,0;0,0]_{H_\mu})=\varphi(\Phi(e,0,0;0);0,0)=\varphi(e,\mu;0,0)\] and by the $G$-equivariance of $\Phi$ it is also clear that
\begin{align*}
\mathcal{T}_0(g'\cdot [g,\nu,\lambda;a,b]_{H_\mu})&=\mathcal{T}_0([g'g,\nu,\lambda;a,b]_{H_\mu})=\varphi(\Phi(g'g,\tilde \nu,\lambda;a\diamond_\l b);a,b)=\\
&=\varphi(g'\cdot \Phi(g,\tilde \nu,\lambda;a\diamond_\l b);a,b)=g'\cdot \varphi(\Phi(g,\tilde \nu,\lambda;a\diamond_\l b);a,b)=\\ &=g'\cdot \mathcal{T}_0([g,\nu,\lambda;a,b]_{H_\mu}).
\end{align*}

We will divide the rest of the proof in three steps. In the first one we prove that there is a set  $G\times_{H_\mu}\left(\p^*_\text{dom}\times \o_\text{dom} \times (T^*S)_\text{dom}\right)$ such that the map $\mathcal{T}_0$ is well defined, it pulls-back the natural symplectic form of $T^*(G\times_H S)$ to the MGS form $G\times_{H_\mu}\left(\p^*_\text{dom}\times \o_\text{dom} \times (T^*S)_\text{dom}\right)$ and it is a local diffeomorphism. In the second one we will show that it is injective in a certain subset and in the third we will prove that it is a diffeomorphism onto its image.

\paragraph{1- $\mathcal{T}_0$ is a local symplectomorphism:\\}

Let $N_0=\o\times S \times S^*$ be the symplectic slice at $z_0=\varphi(e,\mu,0,0)$.
As in Section \ref{sec:reduced-MGS}, there must be an $H_\mu$-invariant neighborhood $(\g_\mu^*)_r$ such that the product $Z:=G\times (\g_\mu^*)_r\times (\o\times S\times S^*)$ with   $\omega_Z:=\omega_{T_\mu}+\Omega_{N_0}$ is a symplectic manifold with $G^L$ and $H_\mu^T$ Hamiltonian actions with momentum maps $\mathbf{K}_{G^L}$ and $\mathbf{K}_{H_\mu^T}$ (see \eqref{eq:Zmomentummaps}).

We will now use the restricted $G$-tube (see Definition \ref{def:restrictedtube}) $\Phi:G\times U_\Phi\subset G\times \g_\mu^*\times \o\times \l^*\to T^*G$
to relate $Z$ with $T^*(G\times S)$. As $\Phi$ is only defined on $G\times U_\Phi$
we will define the open set $$D:=\{(\nu,\lambda,a,b)\mid (\nu,\lambda,a\diamond_\l b)\in U_\Phi,\quad \nu\in (\g_\mu^*)_r\}\subset \g_\mu^*\times \o\times S\times S^*$$ and the map
\begin{align}
f:G\times D &\longrightarrow T^*G \times T^*S  \nonumber\\
(g,\nu,\lambda,a,b) &\longmapsto (\Phi(g,\nu,\lambda,a\diamond_\l b),a,b). \label{eq:f-function}
\end{align}
The pullback of $\omega_{T^*(G\times S)}$ by $f$ is $\omega_Z$, because
\begin{align*}
(f^*\omega_{T^*(G\times S)})(u_1,u_2) &=(\Phi^*\omega_{T^*G})(g,\nu,\lambda,a\diamond_\l b)(v_1,v_2)+\omega_{T^*S}(a,b)(w_1,w_2)=\\
 &= \underbrace{\lpair \dot \nu_2,\xi_1\rpair
  -\lpair \dot \nu_1,\xi_2\rpair
  +\lpair \nu+\mu,[\xi_1,\xi_2]\rpair}_{\omega_{T_\mu}}-\underbrace{\lpair\mu,[\dot\lambda_1,\dot\lambda_2]\rpair+\lpair \dot b_2,\dot a_1\rpair -\lpair b_1,\dot b_1\rpair}_{\Omega_{N_0}}
\end{align*}
where $u_i=(\xi_i,\dot \nu_i,\dot \lambda_i,\dot a_i,\dot b_i)\in T_{(g,\nu,\lambda,a,b)} (G\times D)$.

Note that on $G\times D$ there is a $G^L\times H_\mu^T$ action but on $T^*G\times T^*S\cong G\times \g^*\times S\times S^*$ there is a $G^L\times H^T$ action. As the map $f$ is $G^L\times H_\mu^T$-equivariant it preserves the $H_\mu$-momentum, that is, $\mathbf{K}_{H_\mu^T}=\J_{H_\mu^T}\circ f$. In particular $f(\mathbf{K}_{H_\mu^T}^{-1}(0))\subset \J_{H_\mu^T}^{-1}(0)$. However, the $\l$-momentum property (see Definition \ref{def:restrictedtube}) of restricted $G$-tubes allows us to improve this since for any $\xi\in \l$
\begin{align*}
  \langle\J_{H^T}(f(g,\nu,\lambda;a,b)),\xi \rangle &= \langle \J_R(\Phi(g,\nu,\lambda;a\diamond_\l b))+a\diamond_\h b,\xi\rangle=\\
  &= \langle \J_R(\Phi(g,\nu,\lambda;a\diamond_\l b))\rrestr{\l}+a\diamond_\l b,\xi\rangle= \\
  &=\langle -a\diamond_\l b+a\diamond_\l b,\xi\rangle=0.
\end{align*}
This means that $f$ can be restricted to a map
\begin{equation*}\widetilde{f}:\mathbf{K}_{H_\mu^T}^{-1}(0)\longrightarrow \J^{-1}_{H^T}(0) 
\end{equation*}
and this is the key condition that will allow us to relate the $H_\mu$-quotient $G\times_{H_\mu} \left(\p^*\times N_0\right)$ with the $H$-quotient $\J_{H^T}^{-1}(0)/H^T\cong T^*(G\times_H S)$.
 To do so consider the diagram
 \[\xymatrix{
  & G\times D \ar[r]^{f} & T^*(G\times S) & \\
 G\times \p^*_\text{dom}\times \o_\text{dom}\times (T^*S)_\text{dom} \ar[r]^(.6){l} \ar[d] & *+<10pt>{\mathbf{K}^{-1}_{H_\mu^T}(0)} \ar[r]^{\widetilde{f}} \ar[d] \ar[dr] \ar@{^{(}->}[u] & *+<10pt>{\J^{-1}_H(0)} \ar[d] \ar@{^{(}->}[u] \ar[dr]^\varphi \\
  G\times_{H_\mu} \left(\p^*_\text{dom}\times \o_\text{dom}\times (T^*S)_\text{dom}\right) \ar[r]^(.6)L \ar@/_22pt/[rrr]^{\mathcal{T}_0} & \mathbf{K}^{-1}_{H_\mu^T}(0)/H_\mu^T \ar[r]^{F} &  \J^{-1}_{H^T}(0)/H^T \ar[r]^{\overline{\varphi}} & T^*(G\times_H S)
 }\]

Composing $\widetilde{f}$ with the projection by $H^T$ in the target we get a smooth map $\mathbf{K}_{H_\mu^T}^{-1}(0)\longrightarrow \J^{-1}_{H^T}(0)/H^T$
which is $G^K$-equivariant and $H_\mu^T$-invariant and so it induces the smooth mapping
\[F: \mathbf{K}_{H_\mu^T}^{-1}(0)/H_\mu^T\longrightarrow \J^{-1}_{H^T}(0)/H^T.\]

If $\mathbf{K}_{H_\mu^T}^{-1}(0)/H_\mu^T$ is endowed with the reduced form $(\omega_Z)_\text{red}$ and $\J^{-1}_{H^T}(0)/H^T$ with $(\omega_{T^*(G\times S)})_\text{red}$ then $F^*(\omega_{T^*(G\times S)})_\text{red}=(\omega_Z)_\text{red}$ because $f^*\omega_{T^*(G\times S)}=\omega_Z$. In particular $F$ is an immersion.
Also, as the $H_\mu^T$-action on $G\times D$ is free
 \begin{align*}
 \dim{\mathbf{K}^{-1}_{H_\mu^T}(0)/H_\mu^T}&=\dim \mathbf{K}^{-1}_{H_\mu^T}(0)- \dim \h_\mu=\dim (G\times (\g_\mu^*\times\o))+2\dim S-2\dim \h_\mu \\
 &=\dim \g +\dim \g_\mu +\dim \o -2\dim \h_\mu +2\dim S\\
 &=2\dim \p +2\dim \o+\dim \l +\dim \n +2\dim S\\
 &=2\dim \p +2\dim \o+2\dim \l +2\dim S\\
 &=2(\dim \g-\dim \h+\dim S).
 \end{align*}
Analogously,
  \begin{align*}
 \dim{\J^{-1}_{H}(0)/H^T}&=\dim \J^{-1}_{H^T}(0)- \dim \h=2(\dim \g-\dim \h+\dim S).
 \end{align*}
This implies that $F$ is a local diffeomorphism because is an immersion between spaces of the same dimension.

By continuity we can choose $H_\mu$-invariant neighborhoods of the origin $\p^*_\text{dom}\subset \p^*$, $\o_\text{dom}\subset \o$ and an $H$-invariant neighborhood of the origin $\h_\text{dom}^*\subset \h^*$ such that $(\nu+\underbrace{\frac{1}{2}\lambda \diamond_{\h_\mu} \ad_\lambda^*\mu+a\diamond_{\h_\mu} b}_{\J_{N_0}},\lambda,a,b)\in D$ for any $\nu\in \p^*_\text{dom}$, $\lambda\in \o_\text{dom}$ and $a,b\in T^*S$ with $a\diamond_\h b\in \h_\text{dom}^*$. The map $$L:G\times_{H_\mu} \left(\p^*_\text{dom}\times \o_\text{dom}\times (T^*S)_\text{dom}\right) \to \mathbf{K}^{-1}_{H_\mu^T}(0)/H^T$$
 given by $[g,\nu,\lambda,a,b]_{H_\mu}\mapsto [g,\nu+\J_{N_0}(\lambda,a,b),\lambda,a,b]_{H_\mu}$ is well defined and, as in \eqref{eq:Lmap}, $L^*(\omega_Z)_\text{red}=\Omega_Y$. The conclusion of this first step is that the composition $\mathcal{T}_0:=\bar{\varphi}\circ F \circ L$ is then a local diffeomorphism that pulls-back the canonical form of $T^*(G\times_H S)$ to the MGS form on the set $G\times_{H_\mu} \left(\p^*_\text{dom}\times \o_\text{dom}\times (T^*S)_\text{dom}\right)$.

\paragraph{2- $\mathcal{T}_0$ is locally injective\\}

As $\mathcal{T}_0:G\times_{H_\mu} \left(\p^*_\text{dom}\times \o_\text{dom}\times (T^*S)_\text{dom}\right)\to T^*(G\times_H S)$
 is a local diffeomorphism, there is a neighborhood of $[e,0,0,0,0]_{H_\mu}$ such that $\mathcal{T}_0$ is injective on it. Using that $\mathcal{T}_0$ is $G$-equivariant
  and that the action is proper this neighborhood can be chosen to be
   $G$-invariant (see for example the proof of Theorem 2.3.28 in \cite{ortegaratiu}). That is, $\mathcal{T}_0$ will be injective when restricted to the set $G\times_{H_\mu}\left(\p^*_\text{inj}\times \o_\text{inj}\times (T^*S)_\text{inj}\right)$ where $\p^*_\text{inj}\subset \p^*_\text{dom}$, $\o_\text{inj}\subset \o_\text{dom}$ are $H_\mu$-invariant neighborhoods and $(T^*S)_\text{inj}$ is an $H$-invariant neighborhood of 0 on $(T^*S)_\text{dom}$. Note that we can not ensure that  $(T^*S)_\text{inj}$ will be big enough to contain all the points of the form $(0,b)\in T^*S$. This issue will be addressed in the next step.

\paragraph{3- $\mathcal{T}_0$ is injective \\}

In this step, we will define an open set $(T^*S)_r\subset (T^*S)_\text{dom}$ such that the restriction $$\mathcal{T}_0:G\times_{H_\mu} \left(\p^*_\text{inj}\times \o_\text{inj}\times (T^*S)_r\right) \longrightarrow \mathcal{T}_0 \left( G\times_{H_\mu} \left(\p^*_\text{inj}\times \o_\text{inj}\times (T^*S)_r\right)\right)$$ is a proper map

The key result that we will use
to prove the properness of $\mathcal{T}_0$ is  the following topological result.
\begin{prop}[Lemma 5 of \cite{MR1958531}]
Let $H$ be a Lie group acting on a symplectic vector space $(W,\omega_W)$ denote by $\J:W\to \h^*$  the associated homogeneous momentum map \eqref{linearmomentum}. Then $\J$ is $H$-open relative to its image. That is, if $U$ is an $H$-invariant open set of $W$ then $\J(U)$ is an $H$-invariant open set of the topological space $\J(W)\subset \h^*$.
\label{prop:openmomentum}
\end{prop}

Let $U_1\subset S$ and $U_2\subset S^*$ be $H$-invariant neighborhoods of the origin such that $\overline{U_1\times U_2}\subset (T^*S)_\text{inj}$.
Using Proposition \ref{prop:openmomentum} there is $\h_r^*$ an open neighborhood of $0\in \h^*$ such that
\[\h_r^*\cap (S\diamond_\h S^*)= U_1\diamond_\h U_2\subset \h^*.\]
In this setting, define $(T^*S)_r:=\{(a,b)\in T^*S\mid a\diamond_\h b\in \h_r^*\}$.
 From the first part of the proof we have the following commutative diagram
\[\xymatrix{
G\times \p^*_\text{dom}\times \o_\text{dom}\times (T^*S)_\text{dom} \ar[d]^{\pi_{H_\mu}} \ar[r]^(.6){\widetilde{f}\circ l}                &\J^{-1}_{H^T}(0) \ar[d]^{\overline{\varphi}\circ \pi_H} \\
G\times_{H_\mu} \left(\p^*_\text{dom}\times \o_\text{dom}\times (T^*S)_\text{dom}\right) \ar[r]^(.6){\mathcal{T}_0}    &T^*(G\times_H S) \\
}\]
The problem is that $\widetilde{f}\circ l$ is an injective embedding but it is not clear if it is proper.
We will now show that $\mathcal{T}_0\circ \pi_{H_\mu}$ is a proper map onto its image when restricted to $G\times \p^*_\text{inj}\times \o_\text{inj}\times (T^*S)_r$. To do so let $x_n=(g_n,\nu_n,\lambda_n;a_n,b_n)$ be a sequence in $G\times \p^*_\text{inj}\times \o_\text{inj}\times (T^*S)_r$ such that
\[\mathcal{T}_0(\pi_{H_\mu}(x_n)) \longrightarrow \mathcal{T}_0(\pi_{H_\mu}(\bar{g},\bar{\nu},\bar{\lambda};\bar{a},\bar{b}))\]
with $(\bar{g},\bar{\nu},\bar{\lambda};\bar{a},\bar{b})\in G\times \p^*_\text{inj}\times \o_\text{inj}\times (T^*S)_r$. We will construct a subsequence $\{x_{\sigma_3(n)}\}\subset \{x_n\}$ which is convergent on $G\times \p^*_\text{inj}\times \o_\text{inj}\times (T^*S)_r$.

The map $\overline{\varphi}\circ \pi_H: \J^{-1}_{H^T}(0) \to T^*(G\times_H S)$ is proper because it is a composition of an homeomorphism and the projection by a compact group. Since $\mathcal{T}_0\circ\pi_{H_\mu}=\overline{\varphi}\circ \pi_{H}\circ \widetilde{f}\circ l$ then there is a increasing map $\sigma_1:\N \to \N$ such that the sequence $\{(\widetilde{f}\circ l)(x_{\sigma_1(n)})\}_n$ converges in $\J^{-1}_{H^T}(0)\subset T^*(G\times S)$ (we are just taking a subsequence). But then by uniqueness of the limit there is $h\in H$ such that
\[(\widetilde{f}\circ l)(x_{\sigma_1(n)}) \longrightarrow h\cdot^T ((\widetilde{f}\circ l)(\bar{g},\bar{\nu},\bar{\lambda};\bar{a},\bar{b})),\]
but using the expression of $f$ \eqref{eq:f-function} this implies that $a_{\sigma_1(n)}\to h\cdot \bar{a}$ and $b_{\sigma_1(n)}\to h\cdot \bar{b}$.
By the definition of $(T^*S)_r$
we can choose for each $n$ a pair $(\alpha_n,\beta_n)\in U_1\times U_2$ satisfying \[\alpha_n\diamond_\h \beta_n=a_n\diamond_\h b_n.\]
Since $U_1\times U_2$ is a relatively compact subset of $(T^*S)_\text{inj}$ we can find an increasing map $\sigma_2:\N \to \N$ such that $\sigma_2(\N)\subset \sigma_1(\N)$ and $(\alpha_{\sigma_2(n)},\beta_{\sigma_2(n)}) \to (\alpha_\infty,\beta_\infty)$ but then $(\widetilde{f}\circ l)(g_{\sigma_2(n)},\nu_{\sigma_2(n)},\lambda_{\sigma_2(n)},\alpha_{\sigma_2(n)},\beta_{\sigma_2(n)})$ converges and
\[(\widetilde{f}\circ l)(g_{\sigma_2(n)},\nu_{\sigma_2(n)},\lambda_{\sigma_2(n)},\alpha_{\sigma_2(n)},\beta_{\sigma_2(n)})
\longrightarrow
h\cdot ^T\left((\widetilde{f}\circ l)(\bar{g},\bar{\nu},\bar{\lambda},h^{-1}\alpha_\infty,h^{-1}\beta_\infty)\right).\]

As $(g_{\sigma_2(n)},\nu_{\sigma_2(n)},\lambda_{\sigma_2(n)},\alpha_{\sigma_2(n)},\beta_{\sigma_2(n)})$ lies in $G\times \p^*_\text{inj}\times \o_\text{inj}\times (T^*S)_\text{inj}$ by the facts that $\mathcal{T}_0$ restricted to $G\times \p^*_\text{inj}\times \o_\text{inj}\times (T^*S)_\text{inj}$ is a diffeomorphism and $\pi_{H_\mu}$ is proper then there is an increasing map $\sigma_3:\N\to \N$ with $\sigma_3(\N)\subset \sigma_2(\N)$
such that $(g_{\sigma_3(n)},\nu_{\sigma_3(n)}, \lambda_{\sigma_3(n)}, \alpha_{\sigma_3(n)}, \beta_{\sigma_3(n)})$ converges in $G\times\p^*_\text{inj}\times \o_\text{inj}\times (T^*S)_\text{inj}$. Therefore, $\{x_{\sigma_3(n)}\}$ is a convergent sequence on $G\times \p^*_\text{inj}\times \o_\text{inj}\times (T^*S)_r$.

This proves that $\mathcal{T}_0\circ \pi_{H_\mu}:G\times \p^*_\text{inj}\times \o_\text{inj}\times (T^*S)_r\to (\mathcal{T}_0\circ \pi_{H_\mu})(G\times \p^*_\text{inj}\times \o_\text{inj}\times (T^*S)_r)$ is proper. But since $\pi_{H_\mu}:G\times \p^*_\text{inj}\times \o_\text{inj}\times (T^*S)_r\to G\times_{H_\mu} \left(\p^*_\text{inj}\times \o_\text{inj}\times (T^*S)_r\right)$ is surjective and continuous this implies that $\mathcal{T}_0:G\times_{H_\mu} \left(\p^*_\text{inj}\times \o_\text{inj}\times (T^*S)_r\right)\longrightarrow \mathcal{T}_0\big( G\times_{H_\mu} \left(\p^*_\text{inj}\times \o_\text{inj}\times (T^*S)_r\right) \big)$ is a proper map.

As $\mathcal{T}_0$ is a local homeomorphism which is also proper it is a covering map. Then if $\varphi(e,\mu,0,0)$ has only one preimage, this implies that the covering map is in fact everywhere injective and therefore a global diffeomorphism. But if $\mathcal{T}_0([g,\nu,\lambda,a,b]_{H_\mu})=\varphi(e,\mu,0,0)$ then it is clear from the expression of $\mathcal{T}_0$ \eqref{eq:alpha0tube}
 that $a=b=0$ and then $[g,\nu,\lambda,0,0]_{H_\mu}\in G\times_{H_\mu} \left(\p^*_\text{inj}\times \o_\text{inj}\times (T^*S)_\text{inj}\right)\subset G\times_{H_\mu}  \left(\p^*_\text{inj}\times \o_\text{inj}\times (T^*S)_r\right)$. Therefore, by injectivity  we have $[g,\nu,\lambda,0,0]_{H_\mu}=[e,0,0,0,0]_{H_\mu}$. To sum up, the restricted map
\begin{align*}\mathcal{T}_0:G\times_{H_\mu}  \left(\p^*_\text{inj}\times \o_\text{inj}\times (T^*S)_r\right) &\longrightarrow \mathcal{T}_0(G\times_{H_\mu}  \left(\p^*_\text{inj}\times \o_\text{inj}\times (T^*S)_r\right))\subset T^*(G\times_H S) \\
[g,\nu,\lambda;a,b]_{H_\mu} &\longmapsto \varphi(\Phi(g,\tilde \nu,\lambda;a\diamond_\l b);a,b)
\end{align*}
where $\tilde \nu=\nu+\J_{N_0}(\lambda,a,b)=\nu+\frac{1}{2}\lambda \diamond_{\h_\mu} \ad_\lambda^*\mu+a\diamond_{\h_\mu} b$,
 is a bijection.

\end{proof}

\subsection{The $\Gamma$ map}

In this section we are going to introduce the $\Gamma$ map, a technical tool used in \cite{MR2293645} to build the Hamiltonian cotangent tube when $G=G_\mu$. Here we will use it as the final step towards generalizing the previous tube at $\alpha=0$ to the general case $\alpha\neq 0$.

Let $\varphi(e,\mu,0,\alpha)\in T^*(G\times_H S)$ and define $K=H_\mu \cap H_\alpha$. Recall that in \eqref{eq:Talpha} we defined $B:=(\h_\mu\cdot \alpha)^\circ\subset S$ and a $K$-invariant complement $\s$ of $\k$ in $\h_\mu$. As $K$ is compact we can choose a $K$-invariant splitting $S=B\oplus C$ inducing  the $K$-invariant splitting $S^*=B^*\oplus C^*$.

However the previous splitting of $S^*$ is not in principle  $H_\mu$-invariant. The following technical result studies how it behaves with respect to the  infinitesimal $H_\mu$-action on $S$. The next result is a straightforward generalization to the case $\g_\mu \neq \g$ of Lemmas 27 and 28 of \cite{MR2293645}.

\begin{prop}
In the above situation:
\begin{itemize}
 \item If $a\in B$, $c\in C$ and $b\in B^*$ then\begin{equation*}
  (a+c)\diamond_{\h_\mu} (\alpha+b)=a \diamond_{\h_\mu} b + c\diamond_\s (\alpha+b).
 \end{equation*}

 \item There is a $K$-invariant neighborhood $(B^*)_r$ of the origin in $B^*$ and a $K$-equivariant map \[\Gamma:\s^*\times (B^*)_r\longrightarrow S\] defined by
\begin{equation*}
 \lpair \Gamma(\nu;b),\xi\cdot(\alpha+b)+\beta\rpair=\lpair \nu,\xi\rpair \quad \forall \beta\in B^*,\ \forall \xi\in \s. 
\end{equation*}
That is, $\Gamma$ satisfies $\Gamma(\nu;b)\diamond_\s(b+\alpha) = \nu$ and $\Gamma(\nu;b)\in C$ for any $\nu\in \s^*$ and $b\in (B^*)_r$.
\end{itemize}
\label{prop:Gamma}
\end{prop}

With the notation that we have already introduced the symplectic slice at $\varphi(e,\mu,0,0)$ is $N_0=\o\times T^*S$ whereas the symplectic slice at $\varphi(e,\mu,0,\alpha)$ is $N_\alpha=\o \times T^*B$ (see Theorem \ref{thm:symplnorm-cotang}). The abstract MGS models at $\varphi(e,\mu,0,0)$ and $\varphi(e,\mu,0,\alpha)$ are $G\times_{H_\mu}(\p^*\times N_0)$ and $G\times_{K}(\s^*\oplus \p^*\times N_\alpha)$ respectively. The next result shows that $\Gamma$ can be used to build a well-behaved map between both spaces.

\begin{thm}
In the above context there is an open $K$-invariant neighborhood  $W$ of zero in $(\s^*\oplus \p^*)\times \o \times B \times B^*$ such that the $G$-equivariant map
\begin{align*}
\mathcal{F}: G\times_{K} W &\longrightarrow G\times_{H_\mu}(\p^*\times \o \times S \times S^*) \\
[g,\nu_\s+\nu_\p,\lambda,a,b]_{K} &\longmapsto [g,\nu_\p,\lambda,\widetilde{a},b+\alpha]_{H_\mu},
\end{align*}
where $\widetilde{a}=a+\Gamma(\nu_\s-a\diamond_\s b-\frac{1}{2}\lambda\diamond_\s\ad^*_\lambda \mu;b)$, is a local symplectomorphism.
\label{thm:shift}
\end{thm}

\begin{proof}
As in the first part of the proof of Theorem \ref{thm:tube0} there is a neighborhood $(\g_\mu^*)_r$ of $0\in \g^*$
 such that
 $Z_0:=G\times (\g_\mu^*)_r\times (\o\times S\times S^*)$ is a symplectic space
 with $\omega_{Z_0}:=\omega_{T_\mu}+\Omega_{N_0}$. We are in the same setting as in Subsection \ref{sec:reduced-MGS}, therefore, $Z_0$ supports $G^L$ and $H_\mu^T$ Hamiltonian actions with momentum maps that were denoted as $\mathbf{K}_{G^L}$ and $\mathbf{K}_{H_\mu^T}$.

  Similarly $Z_\alpha:=G\times (\g_\mu^*)_r\times (\o\times B\times (B^*)_r)$ is a symplectic space with symplectic form $\omega_{Z_\alpha}:=\omega_{T_\mu}+\Omega_{N_\alpha}$ because $N_\alpha=\o \times B\times B^*$. Note that $Z_\alpha$ has $G^L$ and $K^T$ Hamiltonian actions with momentum maps $\mathbf{M}_{G_L}$ and $\mathbf{M}_{K^T}$. Consider now the map
\begin{align*}
f:Z_\alpha & \longrightarrow Z_0 \\
(g,\nu,\lambda;a,b) &\longmapsto (g,\nu,\lambda;a+\Gamma(\eta;b),b+\alpha)
\end{align*}
where $\eta=\nu \rrestr{\s}-a\diamond_\s b-\frac{1}{2}\lambda \diamond_\s \ad^*_\lambda \mu$. As $\Gamma$ is $K$-equivariant then  $f$ is $G^L\times K^T$ equivariant. Note that the potential $\theta_{Z_0}(g,\nu,\lambda;a,b)(\xi,\dot \nu,\dot \lambda;\dot a,\dot b) =\lpair \nu+\mu,\xi \rpair+\frac{1}{2}\lpair \mu,[\lambda,\dot\lambda]\rpair+\lpair b, \dot a\rpair-\lpair \alpha,\dot a\rpair$ generates the symplectic structure $\omega_{Z_0}$ (see \eqref{eq:Tmu-potential}) and
\begin{align*}
(f^*\theta_{Z_0})(g,\nu,\lambda,a,b)\cdot v &=\theta_{Z_0}(f(g,\nu,\lambda,a,b))(T_{(g,\nu,\lambda,a,b)}f\cdot v) = \\
&= \lpair \mu + \nu,\xi\rpair +\frac{1}{2}\lpair \mu,[\lambda,\dot\lambda]\rpair +\lpair b,\dot a +T_{(\nu,\lambda,a,b)}\Gamma\cdot(\dot\nu,\dot \lambda,\dot a,\dot b)\rpair=\\
&= \lpair \mu + \nu,\xi\rpair +\frac{1}{2}\lpair \mu,[\lambda,\dot\lambda]\rpair +\lpair b,\dot a\rpair
\end{align*}
where $v=(\xi,\dot \nu,\dot a,\dot b)\in T_{(g,\nu,a,b)} Z_\alpha$. Taking the exterior derivative of this equality we get $f^*\omega_{Z_0}=\omega_{Z_\alpha}$.

Additionally,
the $H_\mu^T$-momentum evalueated at $f(g,\nu,\lambda,a,b)$ is
\begin{align}
\mathbf{K}_{H_\mu^T}(f(g,\nu,\lambda;a,b)) & = - \nu + \frac{1}{2}\lambda \diamond_{\h_\mu} \ad_\lambda^*\mu +(a+\Gamma(\eta;b))\diamond_{\h_\mu} (b+\alpha) = \label{eq:diamondGamma}\\
& = - \nu\rrestr{\h_\mu} + \frac{1}{2}\lambda \diamond_{\h_\mu} \ad_\lambda^*\mu + a\diamond_{\h_\mu} b + \Gamma(\eta;b)\diamond_{\s} (b+\alpha)= \nonumber\\
& = - \nu\rrestr{\h_\mu} + \frac{1}{2}\lambda \diamond_{\h_\mu} \ad_\lambda^*\mu + a\diamond_{\h_\mu} b + \eta= \nonumber\\
& =  - \nu\rrestr{\h_\mu} + \frac{1}{2}\lambda \diamond_{\h_\mu} \ad_\lambda^*\mu + a\diamond_{\h_\mu} b + \nu \rrestr{\s}-a\diamond_\s b-\frac{1}{2}\lambda \diamond_\s \ad^*_\lambda \mu \nonumber\\
&=  - \nu\rrestr{\k} +\frac{1}{2}\ad_\lambda^*\mu \diamond_{\k} \lambda + a\diamond_{\k} b. \nonumber
\end{align}
This means that if $\xi\in \s$ then $\lpair \mathbf{K}_{H_\mu^T}(f(g,\nu,\lambda,a,b)),\xi\rpair=0$ and as $$\mathbf{M}_{K^T}(g,\nu,\lambda,a,b)= -\nu \rrestr{\k}+\frac{1}{2}\ad_\lambda^*\mu \diamond_{\k} \lambda + a\diamond_{\k} b,$$ $f$ can be restricted to
\[\widetilde{f}: \mathbf{M}^{-1}_{K^T}(0) \longrightarrow \mathbf{K}^{-1}_{H_\mu^T}(0).\]

As we did in Theorem \ref{thm:tube0} we can construct from the top to the bottom all the arrows of the diagram
 \[\xymatrix{
  Z_\alpha \ar[r]^(.6){f} & Z_0 & \\
 *+<10pt>{\mathbf{M}^{-1}_{K^T}(0)} \ar[r]^{\widetilde{f}} \ar[d] \ar[dr] \ar@{^{(}->}[u] & *+<10pt>{\mathbf{K}^{-1}_{H_\mu}(0)} \ar[d] \ar@{^{(}->}[u] \\
 \mathbf{M}^{-1}_{K^T}(0)/K^T \ar[r]^{F} &  \mathbf{K}^{-1}_{H_\mu}(0)/{H_\mu^T}
 }\]
using the same arguments as in the first part of Theorem \ref{thm:tube0}. Therefore, as $F$ is an immersion between spaces of the same dimension, it is a local diffeomorphism onto its image.

Adapting the construction of map \eqref{eq:Lmap}  to this setting define the $K$-invariant open set \[W:=\{(\nu_\s+\nu_\p,\lambda,a,b)\in (\s^*\oplus \s^*)\times \o\times B\times (B^*)_r \mid  \nu_\s+\nu_\p+\underbrace{\frac{1}{2}\ad_\lambda^*\mu \diamond_{\k} \lambda + a\diamond_{\k} b}_{\J_{N_\alpha}}\in (\g_\mu^*)_r \}\]
and the map $L_\alpha: G\times_K W \longrightarrow \mathbf{M}^{-1}_{K^T}(0)/K^T$ given by $$L_\alpha([g,\nu_\s+\nu_\p,\lambda,a,b]_{K}) = [g,\nu_\s+\nu_\p+\J_{N_\alpha}(\lambda,a,b),\lambda,a,b]_{K},$$ then $L_\alpha$ is a $G$-equivariant symplectomorphism. And similarly,
$$R_0:\mathbf{K}^{-1}_{H_\mu^T}(0)/H_\mu^T \to G\times_{H_\mu}(\p^*\times \o \times S \times S^*)$$ defined by $[g,\nu,\lambda,a,b]_{H_\mu} \mapsto [g,\nu\rrestr{\p},\lambda,a,b]_{H_\mu}$
is a $G$-equivariant symplectomorphism.

Finally, we can conclude that the composition $\mathcal{F}=R_0\circ F\circ L_\alpha: G\times_K W \to G\times_{H_\mu}(\p^*\times \o \times S \times S^*)$ is a $G$-equivariant local diffeomorphism that pulls-back the MGS symplectic form of the target fo the MGS symplectic form of $G\times_K W$.

\end{proof}

\subsection{General tube}
\label{sec:general-tube}

In this section we will deal with the most general situation and  will construct a Hamiltonian tube around an arbitrary point $\varphi(e,\mu,0,\alpha)$. To do so we will use Theorem \ref{thm:tube0} to obtain a Hamiltonian tube around $\varphi(e,\mu,0,0)$ and then we will compose it with the map $\mathcal{F}$ of Theorem \ref{thm:shift}.
The result of this composition will be the desired Hamiltonian tube around $\varphi(e,\mu,0,\alpha)$.

\begin{thm}
Consider the point $z\in T^*(G\times_H S)$ defined by $z=\varphi(e,\mu,0,\alpha)$. Let $\g=\g_\mu\oplus \o\oplus\l\oplus \n$ be an adapted splitting in the sense of Proposition \ref{prop:liealgsplitting} and let $\Phi:G\times U_\Phi \longrightarrow T^*G$ be an associated restricted $G$-tube. Let $\g_\mu=\h_\mu\oplus \p$ be a $H_\mu$-invariant splitting and  $\h_\mu=\g_z\oplus \s$ a $G_z$ invariant splitting. Define $B=(\h_\mu\cdot \alpha)^\circ \subset S$ and let the map $\Gamma:\s^*\times B^*_r\to S$ be the one defined in Proposition \ref{prop:Gamma}.

In these conditions there are small enough ${G_z}$-invariant neighborhoods of zero $ \s^*_r\subset \s^*$, $ \p^*_r\subset \p^*$, $ \o_r\subset \o$, $B_r\subset B$ and $B_r^*\subset B^*$ such that the map
\begin{align}
\mathcal{T}:G\times_{G_z} \left((\s^*_r\oplus \p^*_r)\times \o_r \times B_r\times B^*_r\right) &\longrightarrow T^*(G\times_H S) \label{eq:gentube}\\
[g,\nu_\s+\nu_\p,\lambda;a,b]_{{G_z}} &\longmapsto \varphi(\Phi(g,\tilde \nu,\lambda;\eps);\tilde a,b+\alpha) \nonumber
\end{align}
where
\begin{align*}
\tilde a &=a+\Gamma(\nu_\s-a\diamond_\s b-\frac{1}{2}\lambda \diamond_\s \ad^*_\lambda\mu;b) \\
\tilde \nu &=\nu_\p+\nu_\s+\frac{1}{2}\lambda\diamond_{{\g_z}}\ad_\lambda^* \mu+a \diamond_{{\g_z}} b  \\
\eps &= \tilde a \diamond_\l (b+\alpha)
\end{align*}
is a Hamiltonian tube around the point $z=\varphi(e,\mu,0,\alpha)$.
\label{thm:tubegeneral}
\end{thm}

\begin{proof}
Using Theorem \ref{thm:tube0}, there is a Hamiltonian tube $\mathcal{T}_0:G\times_{H_\mu} V \to T^*(G\times_H S)$ with $V\subset \p^*\times \o \times S \times S^*$ satisfying $\mathcal{T}_0([e,0]_{H_\mu})=\varphi(e,\mu,0,0)$.
By Theorem \ref{thm:shift}, there is a map $\mathcal{F}: G\times_{G_z} W \to G\times_{H_\mu} (\p^*\times \o \times T^*S)$ with $W\subset (\p^*\oplus \s^*)\times \o \times B \times B^*$.

Note that $G\times_{H_\mu} V$ is an open $G$-invariant subset of $G\times_{H_\mu} (\p^*\times \o \times T^*S)$. Since $\mathcal{F}$ is continuous then the preimage of  $G\times_{H_\mu} V$ by $\mathcal{F}$ is open and it contains the point $[e,0,0,0,0]_{G_z}$ because $\mathcal{F}([e,0,0,0,0]_{G_z})=[e,0,0,0,\alpha]_{H_\mu}\subset G\times_{H_\mu} V$ since $0\diamond_\h \alpha=0\in \h_r^*$.
Therefore, we can choose small enough ${G_z}$-invariant neighborhoods of zero $ \s^*_r\subset \s^*$, $ \p^*_r\subset \p^*$, $ \o_r\subset \o$, $B_r\subset B$ and $B_r^*\subset B^*$ such that the composition $\mathcal{T}:=\mathcal{T}_0\circ \mathcal{F}: G\times_{G_z} (\s^*_r\oplus \p^*_r)\times \o_r \times B_r\times B^*_r \to T^*(G\times_H S)$ is well-defined and injective. Using Theorem \ref{thm:tube0} and \ref{thm:shift} we conclude that $\mathcal{T}$ is a Hamiltonian tube around $\varphi(e,\mu,0,\alpha)$.

More precisely, as $\mathcal{F}([g,\nu_\s+\nu_\p,\lambda,a,b]_{G_z})=[g,\nu_\p,\lambda,\widetilde{a},b+\alpha]_{H_\mu}$ with $\widetilde{a}=a+\Gamma(\nu_\s-a\diamond_\s b-\frac{1}{2}\lambda\diamond_\s\ad^*_\lambda \mu;b)$  then $(\mathcal{T}_0\circ\mathcal{F})([g,\nu_\s+\nu_\p,\lambda;a,b]_{G_z})=\varphi(\Phi(g,\widetilde\nu,\lambda;\eps);\widetilde{a},b+\alpha)$ where $\eps=\widetilde{a}\diamond_\l (b+\alpha)$ and
\[\widetilde{\nu} =\nu_\p+\frac{1}{2}\lambda\diamond_{\h_\mu}\ad_\lambda^* \mu+\widetilde{a} \diamond_{\h_\mu} (b+\alpha),\]
but using exactly the same computations that as in \eqref{eq:diamondGamma} we get $\widetilde{a} \diamond_{\h_\mu} (b+\alpha)=a\diamond_{{\g_z}} b+ \nu_\s-\frac{1}{2}\lambda\diamond_\s\ad^*_\lambda \mu$, that is $\widetilde{\nu}=\nu_\p+\nu_\s+\frac{1}{2}\lambda\diamond_{{\g_z}}\ad_\lambda^* \mu+a \diamond_{{\g_z}} b$.

\end{proof}

Note that if we assume that $\mu\in \g^*$ satisfies $\g_\mu=\g$ then $\o=0$ and the Hamiltonian tube $\mathcal{T}$ will be of the form
\begin{align*}
\mathcal{T}:G\times_{G_z} \left((\s^*_r\oplus \p^*_r) \times B_r\times B^*_r\right) &\longrightarrow T^*(G\times_H S)\\
[g,\nu_\s+\nu_\p;a,b]_{G_z} &\longmapsto \varphi(g,\mu+\tilde \nu;\tilde a,b+\alpha) \nonumber
\end{align*}
where
\[\tilde \nu =\nu_\p+\nu_\s+a \diamond_{{\g_z}} b,\quad   \tilde a =a+\Gamma(\nu_\s-a\diamond_\s b;b).\]
This map is the content of Theorem 31 of \cite{MR2293645}. That is, the map $\mathcal{T}$ coincides with the results of \cite{MR2293645} when we restrict to their totally isotropic hypothesis $\g=\g_\mu$. What happens in this case is that the map given by Theorem \ref{thm:tube0} becomes the trivial $\mu$-shift
\begin{align*}
\mathcal{T}_0:G\times_{H_\mu} \left(\p^*\times S\times S^*\right) &\longrightarrow T^*(G\times_H S) \\
[g,\nu;a,b]_{H_\mu} &\longmapsto \varphi(g,\mu+\nu+a\diamond_{\h_\mu} b;a,b). \nonumber
\end{align*}

The other extreme case is be when $\Gamma$ becomes trivial. This will happen for example if $S=0$, which is equivalent to assume  that locally $Q=G/H$. Fix a point $z=\varphi(e,\mu)\in T^*(G/H)$, as $G_z=H_\mu$ then $\s=0$ and as $S=0$ then $B=0$, therefore, according to \eqref{eq:gentube} $\mathcal{T}$ becomes
\begin{align*}
\mathcal{T}:G\times_{H_\mu}\left(\p^*_r\times \o_r\right) &\longrightarrow T^*(G/H) \\
[g,\nu_\p,\lambda]_{H_\mu} &\longmapsto \varphi(\Phi(g,\nu_\p+\frac{1}{2}\lambda\diamond_{\h_\mu}\ad_\lambda^* \mu,\lambda;0)). \nonumber
\end{align*}

\section{A fibered Bates-Lerman Lemma}
\label{sec:Bates-Lerman}

One of the most important consequences of the MGS model is that it provides a local description of the set of points with momentum $\mu$ which is very useful in the theory of singular reduction. This is the content of the following result that appeared in \cite{MR1486529} and previously in \cite{MR1127479} for $\mu=0$.

\begin{prop}
Let $(M,\omega)$ be a symplectic manifold supporting a Hamiltonian $G$-action with momentum map $\J$. Let $m\in M$, $\mu=\J(m)$ and $\mathcal{T}:G\times_{G_m}(\m^*_r\times N_r)\to M$ a Hamiltonian tube around $m$. There is an  open $G_\mu$-invariant neighborhood $U_M$  of $G_\mu \cdot m$ such that
\[U_M\cap \J^{-1}(\mu)=\mathcal{T}(Z)\]
where
\[Z=\{[g,\nu,\lambda,v]_{G_m}\in \mathcal{T}^{-1}(U_M) \mid g\in G_\mu,\ \nu=0,\ \J_N(v)=0  \}.\]
\label{prop:Bates-Lerman}
\end{prop}

Recall that,  in some cases, we saw that the domain of the cotangent bundle Hamiltonian tube is unbounded in the $S^*$ direction (see Theorem \ref{thm:tube0}).
  In this section we will present an important consequence of this which is that for cotangent-lifted actions the open neighborhood $U_M$ can be global in the vertical direction. That is, it will be of the form $\tau^{-1}(\mathcal{U}_Q)$ where $\tau:T^*Q\to Q$ is the natural projection and $\mathcal{U}_Q$ is a neighborhood  in $Q$.

\begin{prop}
Consider the Hamiltonian tube $\mathcal{T}_0:G\times_{H_\mu} (\p_r^*\times \o_r \times (T^*S)_r)\to T^*(G\times_H S)$
of Theorem \ref{thm:tube0} at the point $\varphi(e,\mu,0,0)\in T^*(G\times_H S)$. There is  a $G_\mu$-invariant neighborhood $\mathcal{U}_Q$ of $[e,0]_H\in G\times_H S$ such that
\begin{equation}\tau^{-1}(\mathcal{U}_Q)\cap \J^{-1}(\mu) = \mathcal{T}_0(Z)\label{eq:fiberedBL}\end{equation}
where
\[Z=\{[g,\nu,\lambda,a,b]_{H_\mu}\in \mathcal{T}_0^{-1}(\tau^{-1}(\mathcal{U}_Q)) \mid g\in G_\mu,\ \nu=0,\ \frac{1}{2}\lambda\diamond_{\h_\mu} \ad^*_\lambda\mu+a\diamond_{\h_\mu} b=0  \}.\]
\label{prop:fiberedBL}
\end{prop}

\begin{proof}
As in the previous section, $N_0=\o\times T^*S$
will be the symplectic slice at $\varphi(e,\mu,0,0)$ with the symplectic form \eqref{eq:symplslice}. The Hamiltonian tube puts the momentum map $\J$ in the normal form \eqref{eq:MGS-momentum}, that is $\J\circ \mathcal{T}_0=\J_Y$. We will now proceed as in the proof of Proposition \ref{prop:Bates-Lerman} in \cite{MR1486529} and we will factorize $\J_Y=\gamma\circ \beta$, with\begin{align*}
\beta:G\times_{H_\mu}(\p^*\times N_0) & \longrightarrow G\times_{H_\mu}\g_\mu^*, &  \gamma:G\times_{H_\mu}\g_\mu^* &\longrightarrow \g^* \\
[g,\nu,v]_{H_\mu} &\longmapsto [g,\nu+\J_{N_0}(v)] &[g,\nu]_{H_\mu} &\longmapsto \Ad_{g^{-1}}^*(\mu+\nu).
\end{align*}
Using this factorization it will be easy to describe $\J_Y^{-1}(\mu)$. Note that, since the map $$T_{[e,0]_{H_\mu}}\gamma \cdot (\xi,\dot \nu)=-\ad^*_\xi \mu +\dot \nu$$ is onto, $\gamma$ is a submersion near $[e,0]_{H_\mu}$ but by $G$-equivariance there is a $G$-invariant open set $U_\text{subm}\subset G\times_{H_\mu} \g^*$ where $\gamma$ is a submersion. Therefore, $\gamma^{-1}(\mu)\cap U_\text{subm}$ is a manifold of dimension $\dim G_\mu-\dim H_\mu$. As $G_\mu\times_{H_\mu} \{0\}\subset \gamma^{-1}(\mu)$, $G_\mu\times_{H_\mu} \{0\}$ has to be an open submanifold of $\gamma^{-1}(\mu)\cap U$, that is, there is an open set $U_\text{BL}\subset U_\text{subm}$ with $ G_\mu\times_{H_\mu} \{0\} = \gamma^{-1}(\mu)\cap U_\text{BL}$. By equivariance of $\gamma$ we can assume that $U_\text{BL}$ is $G_\mu$-invariant. Applying $\beta^{-1}$ on the equality $ G_\mu\times_{H_\mu} \{0\} = \gamma^{-1}(\mu)\cap U_\text{BL}$, we get
\begin{equation}\{[g,0,v]_{H_\mu} \in G\times_{H_\mu}(\p^*\times N_0) \mid \J_{N_0}(v)=0\} = \J_Y^{-1}(\mu)\cap \beta^{-1}(U_\text{BL}).\label{eq:BLeq}\end{equation}

In this setting, let $U_G$ be a $G_\mu^L\times H_\mu^R$-invariant neighborhood of $e\in G$, $\p^*_0\subset\p^*$,  $\o_0\subset \o$ $H_\mu$-invariant neighborhoods of zero and $\h_0^*\subset\h^*$ a $H$-invariant neighborhood of zero such that
\[\{[g,\nu,\lambda,a,b]_{H_\mu} \mid g\in U_G,\ \nu\in \p^*_0,\ \lambda\in \o_0,\ a\diamond_\h b \in \h_0^*\} \subset \big(G\times_{H_\mu} (\p_r^*\times \o_r \times (T^*S)_r)\big) \cap \beta^{-1}(U_\text{BL}).\]

Let now $\Phi$ be the restricted tube used in the definition of $\mathcal{T}_0$ (see Theorem \ref{thm:tube0})
and consider the map
\begin{align*}
f:U_G\times  \p^*_0 \times \o_0 \times \h_0^* & \longrightarrow T^*G/H^T \\
(g,\nu,\lambda,\rho) & \longmapsto \pi_{H^T}(\Phi(g,\nu+\frac{1}{2}\lambda\diamond_{\h_\mu}\ad^*_\lambda \mu+\rho\rrestr{\h_\mu},\lambda, \rho\rrestr{\l})).
\end{align*}
Note that this expression is very similar to \eqref{eq:alpha0tube} but we have changed the dependence on $T^*S$ by a dependence on $\h_\mu^*$.  We claim that this map is a submersion at $(e,0,0,0)$. Using the notation of the proof of Proposition \ref{prop:restrtubes}, if $v=(\xi,\dot\nu,\dot\lambda,\dot \eps)\in T_{(e,0,0,0)}(G\times \g_\mu^*\times \o\times \l^*)$,
\begin{align*}T_{(e,0,0,0)}\Phi\cdot v &= T_{(e,0,0,0)}(\Theta\circ \iota_W\circ \psi^{-1})\cdot v=T_{(e,0,0,0)}\Theta\cdot (\xi,\dot\nu,\dot\lambda+\sigma^{-1}\cdot \eps) \\
 &= (\xi+\dot \lambda+\sigma^{-1}\cdot \dot \eps,\dot \nu-\dot \eps+\ad^*_{\dot \lambda}\mu)
\end{align*}
where $\sigma:\n\to \l^*$ is the linear isomorphism $\sigma\cdot \zeta=\ad^*_\zeta\mu\rrestr{\l}$. Applying this result to $f$ we get
\begin{equation}T_{(e,0,0,0)}f\cdot(\xi,\dot \nu,\dot \lambda,\dot \rho)=T_{(e,\mu)}\pi_{H^T}\cdot (\xi+\dot\lambda+\sigma^{-1}\cdot \rho\rrestr{\l},\dot\nu+\ad_{\dot \lambda}^*\mu+\dot \rho).\label{eq:linearf} \end{equation}
But, since the splitting of Proposition \ref{prop:liealgsplitting} induces the dual decomposition $\g^*=\h_\mu^*\oplus \p^*\oplus \o^*\oplus \l^*\oplus \n^*$ each element of $\g^*$ can be expressed as $\dot \rho\rrestr{\h_\mu}+\dot\nu+\ad_{\dot \lambda}^*\mu+\dot\rho\rrestr{\l}+\ad^*_\eta\mu$ for some $\dot\rho\in \h^*$, $\dot\nu\in \p^*$, $\dot \lambda\in \o^*$ and $\eta\in \h$. Finally $\Ker(T_{(e,\mu)}\pi_{H^T})=\{(\eta,\ad^*_\eta\mu)\mid \eta\in \h\}$ and \eqref{eq:linearf} imply that $T_{(e,0,0,0)}f$ is a surjective linear map.

Since $f$ is $G$-equivariant $f$ is a submersion on a neighborhood of $G\cdot (e,0,0,0)$ and since submersions are open maps the image of $f$ contains an open neighborhood of $G\cdot [e,\mu]_H$ so there must exist  a neighborhood $U_{\g^*}$ of $\mu\in\g^*$ such that \begin{equation}\pi_{H}(G\times_H U_{\g^*})\subset f(U_G\times  \p^*_0 \times \o_0 \times \h_0^*). \label{eq:f-surjective}\end{equation}

Define $\mathcal{U}_G:=\{g\in U_G \mid \Ad_g^*\mu\in U_{\g^*}\} \cap \{g\in U_G\mid \Ad_g^*\mu\rrestr{\h}\in \h_0^*\}$ which is an open neighborhood of $e\in G$ and let $\mathcal{U}_Q:=\mathcal{U}_G\times_H S$ which is an open neighborhood of $[e,0]_H\in Q$. We will now check that $\mathcal{U}_Q$ satisfies \eqref{eq:fiberedBL}.
\begin{itemize}
\item $\tau^{-1}(\mathcal{U}_Q)\cap \J^{-1}(\mu) \supset \mathcal{T}_0(Z)$.

This inclusion is trivial because if $[g,0,\lambda,a,b]_{H_\mu}\in Z$ 
then $\mathcal{T}_0([g,0,\lambda,a,b]_{H_\mu})\in \tau^{-1}(\mathcal{U}_Q)$ and $(\J\circ \mathcal{T}_0)([g,0,\lambda,a,b]_{H_\mu})=\J_Y([g,0,\lambda,a,b]_{H_\mu})=\Ad_{g^{-1}}^*(\mu+\J_{N_0}(\lambda,a,b))=\Ad_{g^{-1}}^*\mu=\mu$.

 \item $\tau^{-1}(\mathcal{U}_Q)\cap \J^{-1}(\mu) \subset \mathcal{T}_0(Z)$.

Let $z\in \tau^{-1}(\mathcal{U}_Q)\cap \J^{-1}(\mu)$. Using the cotangent reduction map $\varphi$ (see  \eqref{eq:varphi-Z}),
 there is an element $(g,\nu,a,b)$ such that $\varphi(g,\nu,a,b)=z$, but, as $\tau(z)\in \mathcal{U}_Q$ then $g\in \mathcal{U}_G$. Since $\varphi(g,\nu,a,b)\in \mathbf{J}^{-1}(\mu)$, using \eqref{eq:Zmomentummaps} we have $\Ad_{g^{-1}}^*\nu=\mu$. Additionally, as $(g,\nu,a,b)\in \mathbf{J}^{-1}_{H^T}(0)$, then $\nu\rrestr{\h}=a\diamond_\h b$. Using $\nu=\Ad^*_{g}\mu$ this implies the relation $(\Ad_{g}^*\mu) \rrestr{\h}=a\diamond_\h b$.

As $g\in \mathcal{U}_G$, using the definition of $\mathcal{U}_G$ we have $(g,\Ad^*_g\mu)\in G\times U_{\g^*}$.
Equation \eqref{eq:f-surjective} implies that there is a point $(g',\nu',\lambda,\rho)\in U_G\times  \p^*_0 \times \o_0 \times \h_0^*$ such that $f(g',\nu',\lambda,\rho)=[g,\Ad_g^*\mu]_H$. Therefore there is $h\in H$ such that
\begin{equation}(gh^{-1},\Ad^*_{h^{-1}}\Ad^*_g\mu)=\Phi(g',\nu'+\frac{1}{2}\lambda\diamond_{\h_\mu}\ad^*_\lambda\mu+\rho\rrestr{\h_\mu},\rho\rrestr{\l}).\label{eq:image_eq}\end{equation}

Moreover, using \eqref{eq:simpletubeH_mu_momentum} and \eqref{eq:implicitPhi} it can be checked that the $H_\mu^T$-momentum of a restricted $G$-tube is $\J_{H_\mu^T}(\Phi(g,\nu,\lambda,\eps))=-\nu\rrestr{\h_\mu}+\frac{1}{2}\lambda\diamond_{\h_\mu}\ad^*_\lambda\mu$. Therefore taking the $H_\mu^T$-momentum on the previous equation
\[-(\Ad^*_{h^{-1}}\Ad^*_g\mu)\rrestr{\h_\mu}=-\frac{1}{2}\lambda\diamond_{\h_\mu}\ad^*_\lambda\mu-\rho\rrestr{\h_\mu}+\frac{1}{2}\lambda\diamond_{\h_\mu}\ad^*_\lambda\mu=-\rho\rrestr{\h_\mu}.\]
Now, using item 4. in Definition \ref{def:restrictedtube} we have that   $H^T$-momentum restricted to $\l^*\subset \g^*$ in \eqref{eq:image_eq} becomes the equality
\[-(\Ad^*_{h^{-1}}\Ad^*_g\mu)\rrestr{\l}=-\rho\rrestr{\l}. \]
That is $(\Ad^*_{h^{-1}}\Ad^*_g\mu)\rrestr{\h}=\rho$. But, as $\Ad^*_g\mu\rrestr{\h}=\nu\rrestr{\h}=a\diamond_\h b$, it follows that $\rho=\Ad_{h^{-1}}^*(a\diamond_\h b)=(h\cdot a)\diamond_\h (h\cdot b)$ and therefore
\begin{align*}
\mathcal{T}_0([g',\nu',\lambda,h\cdot a,h\cdot b]_{H_\mu}) &=\varphi(\Phi(g',\nu'+\frac{1}{2}\lambda\diamond_{\h_\mu}\ad^*_\lambda \mu+\rho\rrestr{\h_\mu},\lambda,(h\cdot a)\diamond_\l(h\cdot b);h\cdot a,h\cdot b)\\
&=\varphi(\Phi(g',\nu'+\frac{1}{2}\lambda\diamond_{\h_\mu}\ad^*_\lambda \mu+\rho\rrestr{\h_\mu},\lambda,\rho\rrestr{\l};h\cdot a,h\cdot b) \\
&=\varphi(gh^{-1},\Ad^*_{h^{-1}}\Ad^*_g\mu,h\cdot a,h\cdot b)\\
&=\varphi(g,\Ad^*_g\mu,a,b)=\varphi(g,\nu,a,b).
\end{align*}
Finally, as $g\in \mathcal{U}_G$, $\Ad^*_g\mu\rrestr{\h}\in \h_0^*$ and  $\h_0^*$ is $H$-invariant, then  $(h\cdot a)\diamond_\h(h\cdot b)=\Ad^*_{h^{-1}}(\Ad^*_g\mu\rrestr{\h})\in \h^*_0$. This observation implies that $(g',\nu',\lambda',h\cdot a,h\cdot b)\in \pi_{H_\mu}^{-1}(\beta^{-1}(U_\text{BL}))$. Using the characterization \eqref{eq:BLeq}, $\nu'=0$, $g\in G_\mu$ and $\J_{N_0}(\lambda',h\cdot a,h\cdot b)=0$, that is $[g',0,\lambda',h\cdot a,h\cdot b]_{H_\mu}\in Z$ as we wanted to show.

\end{itemize}

\end{proof}

\begin{rem}
 Note that we have started with a tube centered around $\varphi(e,\mu,0,0)$. In general if we consider a tube around $\varphi(e,\mu,0,\alpha)$ we can not expect  its image to be global in the $B^*$ direction. This is because all the points in the model space $G\times_{G_z} \big( (\s^*_r\oplus \p^*_r)\times \o_r \times B_r\times B^*_r\big)$
must  have $G$-isotropy conjugated to a subgroup of $G_z$.
 From this observation we can conclude that, in general, $(\s^*_r\oplus \p^*_r)\times \o_r \times B_r\times B^*_r$ will not be an open neighborhood containing containing points of the form $(0,0,0,b)$ for arbitrary large $b\in B^*$. Indeed, if  that was true, we would have $(0,0,0,\alpha)\in (\s^*_r\oplus \p^*_r)\times \o_r \times B_r\times B^*_r$ which would imply $\mathcal{T}([e,0,0,0,-\alpha]_{G_z})=\varphi(e,\mu,0,0)$. But this is a point  with $G$-isotropy $H_\mu$ and in general, $G_z\subsetneq H_\mu$, producing a contradiction.

\end{rem}

\section{Explicit examples}
\label{sec:explicitexamples}

In Proposition \ref{prop:simpleGtube} we proved the existence of simple $G$-tubes using  Moser's trick. In this section we will write down  the actual differential equation that must be solved.
We will see that if $\dim \q=2$ then the simple $G$-tube will be a scaling of an exponential map and we will compute it explicitly for $SO(3)$ and $SL(2,\R)$. From the explicit $SO(3)$ restricted tube in Subsection \ref{sec:so3onr3} we will present the Hamiltonian tube for the natural action of $SO(3)$ on $T^*\R^3$ generalizing the final example of \cite{MR2293645} to $\mu \neq 0$.

We will compute explicitly the flow that determines a simple $G$-tube. For that we are going to use the notation of the proof of Proposition \ref{prop:simpleGtube}.

Recall that we constructed the simple $G$-tube as the composition $\Theta=F\circ \Psi_1$ where $$F(g,\nu,\lambda)=(g \exp(\lambda),\Ad_{ \exp(\lambda)}^*(\nu+\mu))$$ (see \eqref{eq:defF}) and $\Psi_1$ is the time-1
flow of the time dependent vector field $X_t$ that satisfies the Moser equation associated with $\theta_t=tF^*\theta_{T^*G}+(1-t)\theta_Y$, that is, $i_{X_t}(-\ed \theta_t)=\frac{\partial\theta_t}{\partial t}$.

Moser's equation can be written explicitly in this case. Using the above expression of $F$ and \eqref{eq:oneform} we have
\begin{align*}
F^*\theta_{T^*G}(g,\nu,\lambda)(\xi,\dot \nu,\dot \lambda) &=\lpair \Ad_{ \exp(\lambda)}^*(\nu+\mu),\Ad_{ \exp(\lambda)}^{-1}\xi+T_eL_{\exp(\lambda)}^{-1}T_\lambda \exp(\dot\lambda)\rpair \\
&= \lpair \nu+\mu,\xi\rpair + \lpair \nu+\mu,\Ad_{\exp(\lambda)}T_eL_{\exp(\lambda)}^{-1}T_\lambda \exp (\dot \lambda)\rpair \\
&= \lpair \nu+\mu,\xi\rpair + \lpair \nu+\mu,T_eR_{\exp(\lambda)}^{-1} T_\lambda \exp (\dot \lambda)\rpair.
\end{align*}
The last term can be expressed a as a series of Lie brackets (see for example \cite{duistermaat2000lie})
\[M(\lambda)\cdot \dot \lambda:=T_eR_{\exp(\lambda)}^{-1} T_\lambda \exp (\dot \lambda)=\sum_{n\ge 0}\frac{1}{(n+1)!}\ad^n_\lambda\dot\lambda.\]
And, in fact, this series is just the pullback of the right Maurer-Cartan form $\varpi^R(g)=T_eR_g^{-1}$ by the restricted exponential $\exp\rrestr{\q}:\q\to G$. Therefore using the Maurer-Cartan relation
\begin{equation}
\begin{split}(\ed M)(X,Y)=\ed(\exp^*\varpi^R)(X,Y)=\exp^*(\ed\varpi^R)(X,Y)=[\exp^*\varpi^R(X),\exp^*\varpi^R(Y)]\\
=[M(X),M(Y)].
  \end{split}\label{eq:M-MC}
\end{equation}
Now, since $\theta_t(g,\nu,\lambda)(\xi,\dot \nu,\dot \lambda)=\langle \mu+\nu,\xi\rangle+t\langle \mu+\nu,M(\lambda)\cdot \dot \lambda \rangle+(1-t)\frac{1}{2}\langle \mu,[\lambda,\dot \lambda]\rangle$ (using \eqref{eq:M-MC}) the exterior derivative $\omega_t=-\ed\theta_t$ simplifies as
\begin{equation*}
 \begin{split}
  \omega_t(g,\nu,\lambda)(\xi_1,\dot \nu_1,\dot \lambda_1)(\xi_2,\dot \nu_2,\dot \lambda_2)=\lpair \dot \nu_2,\xi_1\rpair-\lpair \dot \nu_1,\xi_2\rpair+\lpair \nu+\mu,[\xi_1,\xi_2]\rpair+\\
  +t\lpair \dot \nu_2,M(\lambda)\dot \lambda_1\rpair-t\lpair \dot \nu_1,M(\lambda)\dot \lambda_2\rpair
  +t\lpair \nu+\mu,-[M(\lambda)\cdot \dot \lambda_1,M(\lambda)\cdot \dot \lambda_2]\rpair-(1-t)\lpair \mu,[\dot \lambda_1,\dot \lambda_2]\rpair.
 \end{split}
\end{equation*}
Also, the expression $\frac{\partial \theta_t}{\partial t}=\theta_1-\theta_0$ can be written as 
\begin{equation}\frac{\partial \theta_t}{\partial t}(g,\nu,\lambda)(\xi,\dot \nu,\dot \lambda)=\langle \nu+\mu,M(\lambda)\cdot \dot \lambda\rangle-\langle \mu,\frac{1}{2}\ad_\lambda{\dot\lambda}+\dot\lambda\rangle. \label{eq:d-dt-oneform}\end{equation}

From now on we will assume that $\dim \q=2$. Note that the one-form
\begin{align*}
  \omega_t(g,\nu,\lambda)(0,0,\lambda)(\xi_2,\dot \nu_2,\dot \lambda_2) &= t\lpair \dot \nu_2,\lambda\rpair
  +t\lpair \nu+\mu,-[\lambda,M(\lambda)\cdot \dot \lambda_2]\rpair-(1-t)\lpair \mu,[\lambda,\dot \lambda_2]\rpair \\
  &=t\lpair \nu+\mu,-M(\lambda)[\lambda,\dot \lambda_2]\rpair-(1-t)\lpair \mu,[\lambda,\dot \lambda_2]\rpair
\end{align*}
and \eqref{eq:d-dt-oneform} have the same kernel $\g\oplus \g_\mu^*\oplus \R\cdot \lambda\subset T_{(g,\nu,\lambda)} (G\times \g_\mu^*\times \q)$. Therefore, there is a real-valued function
$f$ such that
\[\omega_t(g,\nu,\lambda)(0,0,f(\nu,\lambda,t)\lambda)(\xi_2,\dot \nu_2,\dot \lambda_2)=\frac{\partial \theta_t}{\partial t}(g,\nu,\lambda)(\xi_2,\dot \nu_2,\dot \lambda_2).\]
That is, $X_t(g,\nu,\lambda)=f(\nu,\lambda,t)\frac{\partial}{\partial\lambda}$ and in particular $\Psi_t(g,\nu,\lambda)=(g,\nu,m_t(\nu,\lambda)\lambda)$ for certain scaling factor $m_t:\g_\mu^*\times \q\to \R$. We will obtain an equation that fully determines $m_1$ and therefore the map $\Psi_1$ and the simple $G$-tube.

Taking the time-derivative of the time-dependent pull-back (see \eqref{eq:dt_time_flow})
\begin{align*}
 \frac{\partial}{\partial t}\left(\Psi^*_t\theta_t\right)&=\Psi^*_t\left( (\ed i_{X_t}+i_{X_t}\ed)\theta_t+\frac{\partial\theta_t}{\partial t} \right)=\Psi^*_t( \ed i_{X_t}\theta_t)=\Psi^*_t( \ed i_{X_t}(\theta_0+t(\theta_1-\theta_0))=\\
 &=\Psi^*_t( \ed i_{X_t}\theta_0)=\ed(\Psi^*_t (i_{X_t} \theta_0)).
\end{align*}
Additionally
\[\Psi^*_t (i_{X_t} \theta_0)=\Psi^*_t(\langle \mu,\lambda\rangle f(\nu,\lambda,t))=\langle \mu,\lambda\rangle f(\nu,m_t(\nu,\lambda)\lambda,t))=\frac{\partial}{\partial t} \langle \mu,m_t(\nu,\lambda) \lambda\rangle\]
from where we get \[\frac{\partial}{\partial t}\left(\Psi^*_t\theta_t-\ed\langle \mu,m_t(\nu,\lambda) \lambda\rangle\right)=0.\]

This equation implies that $\Psi_t$ satisfies the following equation on one-forms
\begin{equation}\Psi^*_1\theta_1-\ed\langle \mu,m_1(\nu,\lambda) \lambda\rangle=\theta_0-\ed\langle \mu,\lambda\rangle. \label{eq:oneform-preservation}\end{equation}
But this equation does not depend on the derivatives of the scaling factor $m_1$ because
\begin{align*}
 \Psi^*_1\theta_1(\xi,\dot \nu,\dot \lambda) &=\langle \mu+\nu,(\D_\nu m_1\cdot \dot \nu+\D_\lambda m_1\cdot \dot \lambda)\lambda+M(m_1\lambda)\cdot(m_1\dot \lambda)\rangle=\\
 &=\langle \mu,\lambda\rangle (\D_\nu m_1\cdot \dot \nu+\D_\lambda m_1\cdot \dot \lambda)+\langle \mu+\nu,M(m_1\lambda)\cdot(m_1\dot \lambda)\rangle\
\end{align*}
and
\[\ed\langle \mu,m_1 \lambda\rangle(\xi,\dot \nu,\dot \lambda)=\langle \mu,\lambda\rangle (\D_\nu m_1\cdot \dot \nu+\D_\lambda m_1\cdot \dot \lambda)+\langle \mu,m_1\dot \lambda\rangle.\]
Since $\langle \mu,[\lambda,\dot \lambda]\rangle$ is a non-vanishing one-form on the two dimensional space $\q$ with kernel $\R\cdot \lambda$
 and $\langle \mu+\nu,M(\lambda)\cdot\dot \lambda\rangle-\langle \mu,\dot \lambda\rangle$ has also kernel $\R\cdot\lambda$ then there is an analytic function $h(\lambda,\nu)$ such that
$$\langle \mu+\nu,M(\lambda)\cdot\dot \lambda\rangle-\langle \mu,\dot \lambda\rangle=h(\lambda,\nu)\langle \mu,[\lambda,\dot \lambda]\rangle.$$ 
With this notation, \eqref{eq:oneform-preservation} becomes the non-linear equation
\begin{equation}h(m_1\,\lambda,\nu) m_1^2=\frac{1}{2}.\label{eq:h-eq}\end{equation}
Its solution $m_1(\lambda,\nu)$  is enough to write down explicitly the simple $G$-tube $$\Theta(g,\nu,\lambda)=(g\exp(m_1(\lambda,\nu)\lambda),\Ad_{\exp(m_1(\lambda,\nu)\lambda)}(\nu+\mu)).$$ In the following lemmas we will see that under some algebraic assumptions on $\g$ we can write down $m_1$ in terms of elementary functions.

\begin{lem}
Assume that the subspace $\q$ defined by the splitting $\g=\g_\mu\oplus \q$ is a 2-dimensional subalgebra.
Then the equation \eqref{eq:h-eq} has the solution $m_1(\nu,\lambda)=\mathcal{E}(-\tr(\ad_\lambda\rrestr{\q}))$ where $\mathcal{E}:\R\to \R^+$ the unique analytic function that satisfies
\begin{equation}e^{-x\mathcal{E}(x)}=1-x\mathcal{E}(x)+\frac{x^2}{2}. \label{eq:Efunction}\end{equation}
\label{lem:sl2-parabolic}
\end{lem}
\begin{proof}
As the dimension of $\q$ is two and $\ad_\xi$ is singular it follows that $\ad_\eta^2\rrestr{\q}-\tr(\ad_\eta\rrestr{\q})\ad_\eta\rrestr{\q}=0$ for any $\eta\in \q$. Therefore,
\[\sum_{k\ge 0} \frac{1}{(k+1)!} \ad_{\xi}^k=\operatorname{Id}+ \sum_{k\ge0}\frac{(\tr(\ad_\xi\rrestr{\q}))^k}{(k+2)!} \ad_\xi =\operatorname{Id}+\frac{e^x-x-1}{x^2}\ad_\xi\]
where $x=\tr(\ad_\xi\rrestr{\q}) $. Then \eqref{eq:oneform-preservation} becomes
\[\langle \mu+\nu,M(\lambda)\cdot \dot \lambda\rangle=\langle \mu,M(\lambda)\cdot \lambda\rangle=\langle \mu,\dot \lambda\rangle+\langle \mu,[\lambda,\dot \lambda]\rangle \frac{e^x-x-1}{x^2}.\]
Comparing with \eqref{eq:h-eq} we see that $h(\nu,\lambda)=\frac{e^{-x}+x-1}{x^2}$ with $x=-\tr(\ad_\xi\rrestr{\q})$.
\end{proof}

\begin{rem}
The function $\mathcal{E}$ can be written in terms of the \emph{Lambert W function} (see \cite{MR1414285})
\[\mathcal{E}(x)=\begin{cases} \frac{x}{2}+\frac{W_0(-\exp(-1-\frac12 x^2))+1}{x} & \text{if } x>0 \\
   \frac{x}{2}+\frac{W_{-1}(-\exp(-1-\frac12 x^2))+1}{x} & \text{if } x<0
  \end{cases}
\]
where $W_0$ and $W_{-1}$ are defined the same reference. It can be checked that $\mathcal{E}(x)$ is positive and strictly increasing for all $x\in \R$. Additionally $\mathcal{E}(x)$ is asymptotic to $\frac{x}{2}$ if $x\to \infty$, and satisfies $\mathcal{E}(0)=1$ and $\mathcal{E}(x)\to 0$ if $x\to -\infty$.
\end{rem}

\begin{lem}
Assume that the splitting $\g=\g_\mu\oplus \q$ satisfies
\begin{enumerate}
 \item $\ad^3_\xi+a(\xi)\ad_\xi=0$ $\forall \xi\in \q$ for a certain smooth function $a:\q\to \R$, and
 \item $\langle \mu+\nu,\ad_\xi^2\eta\rangle=0$ for any $\xi,\eta\in\q$ and $\nu\in\q^\circ$.
\end{enumerate}
In addition, let $b:(\g_\mu^*)_r\to \R$ be the function satisfying $\langle \nu+\mu,[\xi,\eta]\rangle=b(\nu)\langle \mu,[\xi,\eta]\rangle$ for any $\xi,\eta\in \q$. 
Then equation \eqref{eq:h-eq} has the solution $m_1(\nu,\lambda)=\mathcal{F}\left(\frac{a(\lambda)}{4b(\nu)}\right)\frac{1}{\sqrt{b(\lambda)}}$, where $\mathcal{F}:(-\infty,1)\to \R^+$ is the analytic function
\[\mathcal{F}(x)=\begin{cases}
   \frac{\arcsin(\sqrt x)}{\sqrt x} & \textrm{ if } x> 0 \\
   \frac{\operatorname{arcsinh} (\sqrt{|x|})}{\sqrt{|x|}} &\textrm{ if } x< 0
  \end{cases}
\]
\label{lem:so3-computations}
\end{lem}

\begin{proof}
Using the first hypothesis $\displaystyle\sum_{n\ge 0} \frac{1}{(n+1)!} \ad_{\xi}^n=\Id+A_1(a(\xi))\ad_\xi+A_2(a(\xi))\ad_\xi^2$
where $A_1$ and $A_2$ are analytic scalar functions. Then
\[\langle \mu+\nu, M(\lambda)\cdot \dot \lambda)=\langle \mu+\nu,\dot \lambda\rangle +A_1(a(\lambda))\langle \mu+\nu,[\lambda,\dot \lambda]\rangle=\langle \mu,\dot \lambda\rangle +A_1(a(\lambda))b(\nu)\langle \mu,[\lambda,\dot \lambda]\rangle,\]
that is $h(\lambda,\nu)=A_1(a(\lambda))b(\nu)$. It can be checked that $A_1(x)=\frac{1-\cos(\sqrt x)}{x}$. If we assume $a(\lambda)>0$ then using simple formal manipulations \eqref{eq:h-eq} is equivalent to $m_1\sqrt{a(\lambda)}=\operatorname{arccos}\left(1-\frac{a(\lambda)}{2b(\nu)}\right)$ and as $2\operatorname{arcsin} x=\operatorname{arccos}(1-2x^2)$ then $m_1(\lambda,\nu)=\mathcal{F}\left(\frac{a(\lambda)}{4b(\nu)}\right)\frac{1}{\sqrt{b(\lambda)}}$. If $a(\lambda)\le 0$ a similar computation gives the same result.
\end{proof}

\subsection{$SO(3)$ simple tube}

Under the hat map the Lie algebra $\g=\mathfrak{so}(3)$ can be identified with $\R^3$ equipped with the cross product. The standard inner product $\langle\cdot,\cdot \rangle$ on $\R^3\cong \g$ will correspond to the dual pairing between $\g$ and $\g^*$, identifying them.

Fix an element $\mu\in \g^*$. We have two different possibilities:

\begin{itemize}
 \item $\mu=0$. In this case the $G$-tube is trivial (see Remark \ref{rem:isotropicmu}).
 \item $\mu\neq 0$. In this case  $\g_\mu$ is the subspace generated by $\mu$ and we will define $\q$ as the orthogonal complement to $\g_\mu$. The subspace $\g_\mu^*$ being the annihilator of $\q$ is also identified with the subspace generated by $\mu$.

The vector identity $\mathbf{a}\times (\mathbf{a}\times \mathbf{c})=\langle\mathbf{a},\mathbf{c}\rangle \mathbf{a}-\langle\mathbf{a},\mathbf{a}\rangle\mathbf{c}$ implies that both conditions of Lemma \ref{lem:so3-computations} hold for $\mathfrak{so}(3)$ with $a(\lambda)=\|\lambda\|^2$. Therefore the map
\begin{align}
G\times \g_\mu^*\times \q & \longrightarrow SO(3)\times \R^3\cong T^*SO(3) \label{eq:so3tube}\\
(g,\nu,\lambda) & \longmapsto (gE(\nu,\lambda),E(\nu,\lambda)\cdot (\nu+\mu)) \nonumber
\end{align}
with $E(\nu,\lambda)=\exp\left(2\frac{\arcsin\left(\sqrt{\frac{\mu}{\mu+\nu}}\frac{\|\lambda\|}{2}\right)}{\|\lambda\|} \widehat{\lambda} \right)$ is a simple $SO(3)$-tube at $(e,\mu)\in T^*SO(3)$.
\end{itemize}

Note that this expression is exactly the same as the one obtained in Theorem 3 of \cite{1311.7447}. In fact, this map was known in celestial mechanics as \emph{regularized Serret-Andoyer-Deprit coordinates} (see \cite{MR2222429} and references therein).

\subsection{$SL(2,\R)$ simple tube}

On the Lie algebra $\g=\mathfrak{sl}(2,\R)$ the bilinear form $\langle A,B\rangle=-2\text{Tr}(AB)$ is non-degenerate and we will use it to identify $\g$ and $\g^*$. If $\xi,\eta\in \mathfrak{sl}(2,\R)$ it can be checked that $\ad_\xi\ad_\xi \eta=\langle\xi,\eta\rangle\xi-\langle\xi,\xi\rangle\eta$ and then for any $\xi\in\g$ we have $\ad_\xi^3+\|\xi\|^2\ad_\xi=0$.

Fix an element $\mu\in \g^*$. We now have three different cases:

\begin{itemize}
 \item $\mu=0$. In this case the $G$-tube is trivial (see Remark \ref{rem:isotropicmu})
 \item $\|\mu\|^2:=\langle \mu,\mu\rangle\neq 0$. Then $\g_\mu$ is one dimensional and  is the space generated by $\mu$. We will define $\q$ to be the orthogonal space to $\mu$ with respect to the pairing. Since the norm of $\mu$ is non-zero $\g=\g_\mu\oplus \q$. As before $\g_\mu^*=\q^\circ=\g_\mu$ so we can apply Lemma \ref{lem:so3-computations} obtaining that \begin{align*}
G\times \g_\mu^*\times \q & \longrightarrow T^*SL(2,\R) \\
(g,\nu,\lambda) & \longmapsto (gE(\nu,\lambda),\Ad^*_{E(\nu,\lambda)} (\nu+\mu))
\end{align*}
with $E(\nu,\lambda)=\exp\left(\mathcal{F}\left(\frac{\|\lambda\|^2\mu}{4(\mu+\nu)}\right)\sqrt{\frac{\mu}{\mu+\nu}}\lambda\right)$ is a simple $SL(2,\R)$-tube at $(e,\mu)\in T^*SL(2,\R)$.

 \item $\|\mu\|^2=0$ and $\mu\neq 0$. In this case, using basic linear algebra, it can be shown that
 there is $k\in SL(2,\R)$ such that $\mu=k \begin{bmatrix} 0 & s \\ 0 & 0 \end{bmatrix} k^{-1}$ with $s=1$ or $s=-1$.

Also in this case $\g_\mu$ is the subspace generated by $\mu$,  and we will define $\q$ as the subspace generated by $k\begin{bmatrix} 1 & 0 \\ 0 & -1 \end{bmatrix}k^{-1}$ and $k\begin{bmatrix} 0 & 0 \\ 1 & 0 \end{bmatrix}k^{-1}$. A generic element in $\q$ will be represented as $k\begin{bmatrix} a & 0 \\ b & -a \end{bmatrix}k^{-1}$. It can be checked that $\g_\mu^*=\q^\circ$ is the subspace generated by $k\begin{bmatrix} 0 & 0 \\ 1 & 0 \end{bmatrix}k^{-1}$.

A simple computation shows that $\q$ is a subalgebra of $\g$ so we can apply Lemma \ref{lem:sl2-parabolic}, obtaining that the map
\begin{align*}
G\times \g_\mu^*\times \q & \longrightarrow T^*SL(2,\R) \\
(g,\nu\mu,\lambda) & \longmapsto \Big(gE(\lambda),\Ad^*_{E(\lambda)}\big( (\nu+1)\mu\big)\Big),
\end{align*}
where $\lambda=k\begin{bmatrix} a & 0 \\ b & -a \end{bmatrix}k^{-1}$, $\nu\in \R$ and $E(\lambda)=\exp\left(\mathcal{E}(2a) \lambda \right)$, 
 is a simple $SL(2,\R)$-tube at $(e,\mu)\in T^*SL(2,\R)$.
 Note that for this tube the domain is the whole space $G\times \g_\mu^*\times \q$. There are no restrictions on  $\nu$ or $\lambda$ but the map is not onto.

\end{itemize}

\subsection{A $SO(3)$ restricted tube}
Let $H$ be a compact non-discrete subgroup of $SO(3)$. Note that $H$ must be one-dimensional. We will denote by $\xi_\h\in \R^3$ the generator of $\h$ with unit norm. In this setting the adapted splitting of Proposition \ref{prop:liealgsplitting} reduces to $\l=\R\cdot\xi_\h$, $\p=\R\cdot \mu$ and $\n=\R\cdot \xi_\h\times \mu $. To obtain the restricted tube we will use \eqref{eq:implicitPhi}, so we neeed to find $\zeta\in \n$ satisfying the condition
\begin{equation}\J_R(\Theta(g,\nu,\zeta)\rrestr{\l}=-\eps \label{eq:cond1}\end{equation} as a function of $\nu$ and $\eps$. Using the notation of \eqref{eq:simpleGtube-E}, $\Theta$ can be written as $$\Theta(g,\nu,\lambda)=(gE(\nu,\lambda),\Ad_{E(\nu,\lambda)}^*(\nu+\mu)).$$ Using Proposition \ref{prop:T*G} we can rewrite \eqref{eq:cond1} as
\begin{equation}\Ad^*_{E(\nu,\zeta)}(\nu+\mu)\rrestr{\l}=\eps. \label{eq:cond2} \end{equation}
Applying the explicit expression \eqref{eq:so3tube} for the $SO(3)$ simple tube we have $E(\nu,\zeta)=\exp(\rho(\nu,\zeta) \frac{\xi_\h\times \mu}{\|\xi_\h\times \mu\|})$. Then, solving \eqref{eq:cond2} is equivalent to finding the real parameter $\rho$ as a function of $\nu$ and $\eps$ that satisfies

\[\left\langle \exp\Big(-\rho\frac{\xi_\h\times \mu}{\|\xi_\h\times \mu\|}\Big)\cdot (\nu+\mu),\xi_\h\right\rangle = \langle \eps,\xi_\h\rangle.\]
Since $\{\xi_\h,\frac{\mu}{\|\mu\|},\frac{\xi_\h\times \mu}{\|\xi_\h\times \mu\|} \}$ is an orthogonal basis this last equation is equivalent to
$$\langle\sin(\rho)(\nu+\mu), \frac{\mu}{\|\mu\|}\rangle=\langle\eps, \xi_\h\rangle.$$  Therefore, if we denote by $r$ the expression $\arcsin\frac{(\eps\cdot \xi_\h)\|\mu\|}{(\nu+\mu)\cdot \mu}$, the equation
\begin{equation}
\Phi(g,\nu;\eps)=\left(g\exp\left(r \frac{\xi_\h\times \mu}{\|\xi_\h\times \mu\|}\right),\exp\left(-r \frac{\xi_\h\times \mu}{\|\xi_\h\times \mu\|}\right)\cdot (\nu+\mu)\right)\in SO(3)\times \R^3 \label{eq:restrso3tube}
\end{equation}
defines a a restricted $SO(3)$-tube.

\subsection{Hamiltonian tube for $SO(3)$ acting on $T^*\R^3$}
\label{sec:so3onr3}

Consider the natural action of $SO(3)$ on $\R^3$ and fix a point $z=(q,p)\in T^*\R^3$. If $\mu=q\times p=0$ then \cite{MR2293645} provides an explicit computation of the Hamiltonian tube centered at $(q,p)$. Therefore, we will assume $\mu=q\times p\neq 0$, and in particular $q\neq 0$ so the isotropy $H:=G_q$ is the group of rotations with axis $q$. The linear slice $S=(\g\cdot q)^\perp$ is the subspace generated by $q$, and note that this subspace is fixed by $H$. As $\mu$ and $q$ are perpendicular the groups $H_\mu$ and $G_z$ are trivial. Under the identification of $\mathfrak{so}(3)$ with $\R^3$ the linear splitting of Proposition \ref{prop:liealgsplitting} becomes $\l=\R\cdot q$, $\n=\R\cdot (\mu \times q)$ and $\g_\mu=\R\cdot \mu$. Recall that$\alpha:=z\rrestr{S}\in S^*$ (see Theorem \ref{thm:symplnorm-cotang}). In this setting $\alpha=\frac{-\mu\times q}{\|q\|^2}$.

 Theorem \ref{thm:tubegeneral} together with the explicit expression for the restricted tube \eqref{eq:restrso3tube} give that
\begin{align*}
G\times\g_\mu^*\times T^*S &\longrightarrow T^*(SO(3)\times_H S)\\
(g,\nu,a,b) & \longmapsto \varphi(g,\nu+\mu, a, b+\alpha)
\end{align*}
is a Hamiltonian tube at $\varphi(e,\mu,0,\alpha)\in T^*(SO(3)\times_H S)$. In this case the parameter $\eps=a\diamond_\l b$ always vanishes because $S$ is fixed by the whole group $H$.

Let $S_r:=\{a \in \R^3 \mid \langle a,q\rangle=0,\  \|a\|<\|q\|\}\subset S$. It is a standard fact that the map $\mathbf{s}:G\times_H S_r\to \R^3$ defined by $[g,a]_H\mapsto g\cdot (q+a)$ is a Palais' tube around $q$. Using this  tube, after some easy manipulations the previous Hamiltonian tube at $(q,p)$ can be written as
\begin{align*}
\mathcal{T}:G\times_{\Id}\left( \g_\mu^*\times S_r \times S^*\right) &\longrightarrow \R^3\times \R^3 \cong T^*\R^3 \\
(g,\nu,a,b) &\longmapsto \Bigg(g\cdot (q+a),g\cdot \Big( (\nu+\mu)\times \frac{q+a}{\|q+a\|^2}+b+\underbrace{\frac{-\mu\times q}{\|q\|^2}}_\alpha\Big)\Bigg).
\end{align*}

\clearpage





\end{document}